\documentclass[leqno,12pt]{amsart}

\usepackage{graphicx}%
\usepackage{multirow}%
\usepackage{amsmath,amssymb,amsfonts,amsthm}%
\usepackage[numbers]{natbib}
\usepackage{mathrsfs}
\usepackage[hidelinks]{hyperref}
\usepackage{enumitem}
\usepackage{mathtools}%
\usepackage{nicefrac}
\usepackage[title]{appendix}%
\usepackage{xcolor}%
\usepackage{setspace}
\usepackage{bbm}

\DeclareRobustCommand{\bb}[1]{\mathbb{#1}}

\DeclareRobustCommand{\t}[1]{\text{#1}}

\DeclareRobustCommand{\set}[1]{\left\{#1 \right\}}
\DeclareRobustCommand{\Set}[2][]{\left\{#1 \, \middle|\, #2 \right\}}
\DeclareRobustCommand{\bk}[2]{\left\langle #1,\,#2\right\rangle}

\DeclareRobustCommand{\vp}{\varphi}

\DeclareRobustCommand{\Var}{\mathrm{Var}}

\newcommand\restr[2]{{
  \left.\kern-\nulldelimiterspace 
  #1 
  \vphantom{\rule{0pt}{9pt}} 
  \right|_{#2} 
  }}

\theoremstyle{remark}%
\newtheorem*{example}{Example}%
\newtheorem{remark}{Remark}%

\theoremstyle{definition}%
\newtheorem{assumption}{Assumption}

\newtheorem{theorem}{Theorem}%
\newtheorem{proposition}[theorem]{Proposition}%
\newtheorem{lemma}{Lemma}[section]
\newtheorem{corollary}[theorem]{Corollary}

\usepackage[T1]{fontenc}
\usepackage{tgtermes}

\calclayout

\begin{document}

\title[]{A scale-free density bound for Gaussian maxima}

\begin{abstract}\normalsize
We derive a {scale-free} bound on the density of the maximum of a centered Gaussian vector. 
The basic bound is non-uniform, depends logarithmically on the dimension, and allows any covariance matrix. 
When the largest marginal variance is separated from zero, it implies that the density of the maximum is uniformly controlled at all quantiles above $\nicefrac{2}{3}$, which is sufficient for many hypothesis testing applications; it yields validity of Gaussian and bootstrap approximations for maxima of high-dimensional sums at test levels $\alpha < \nicefrac{1}{3}$ without further restricting the covariance. 
Under these same conditions, the argument is extended to show that the maximum absolute value of a Gaussian vector has a uniformly bounded density on the real line.
The result also implies new bounds on the variance of the maximum. 
%
\end{abstract}

\author{Suhas Vijaykumar}
\address{Department of Economics, University of California, San Diego. 9500 Gilman Drive \#0508, La Jolla, CA 92093}
\email{svijaykumar@ucsd.edu}

\keywords{Anti-concentration, maxima, Gaussian approximation, bootstrap}

\hypersetup{
    pdftitle={A scale-free density bound for Gaussian maxima},
    pdfauthor={Suhas Vijaykumar},
    pdfkeywords={Anti-concentration, maxima, Gaussian approximation, bootstrap}
}

\maketitle

\onehalfspacing 

\section{Introduction}\label{sec:intro}

Gaussian and bootstrap approximations for maxima of sums have become an essential component of modern statistics, providing practical methods for inference in problems of high-dimensional selection, multiple testing, and non-parametric estimation. These tools are especially useful when the distribution being approximated has a complex form and its covariance is difficult to characterize, yet simulation-based inference remains effective \citep{chernozhukov2023high}.

A certain limitation of the theory, however, concerns \emph{degeneracy}: settings where coordinates have vanishing variance, a high degree of correlation, and where standardization is not viable. In these cases, it becomes difficult to verify that the approximating law has a bounded density (or is \emph{anti-concentrated}), which is needed to justify conventional statistical inference. This setting arises in a wide range of modern statistical problems, from multiple testing with genomic or gene-expression data, to anytime-valid sequential testing, to the construction of uniform confidence bands when data have latent, low-dimensional structure. 

In each of these cases, the standard remedy of dividing by coordinate-wise standard deviations can fail. This happens for two primary reasons: as $\sigma(x)$ approaches zero, uniform control of $|\hat\sigma(x)/\sigma(x)-1|$ becomes challenging, and, moreover, the standardized process may not be tight. In the context of high-dimensional correlation testing, the relevant ``coordinates'' are sparse linear combinations $\bk{\beta}{X}$, and estimating the variance of each such combination is non-trivial; indeed, certifying that the smallest restricted eigenvalue exceeds any fixed threshold is NP-hard \citep{bandeira2013certifying,dobriban2016regularity}. In anytime-valid sequential testing, standardizing the time-dependent variance results in an approximating Gaussian law that is not tight. In non-parametric inference on data with latent low-dimensional structure, pointwise variances depend strongly on the unknown intrinsic dimension and local geometry, causing the two problems to compound.

 This paper studies anti-concentration and Gaussian approximation under degeneracy. We state density bounds for the maximum which depend logarithmically on the dimension and, importantly, place minimal restrictions on the covariance. This leads to Gaussian and bootstrap approximations under similarly permissive conditions.

The most general bound covers the upper quantiles of the distribution, including the most relevant range for hypothesis testing. As a consequence, we demonstrate validity of standard Gaussian and bootstrap critical values for maxima of high-dimensional sums at levels $\alpha < \nicefrac{1}{3}$ under the simple requirement of a single non-degenerate coordinate.

These arguments further imply variance bounds and global anti-concentration bounds (i.e., ones that hold at all quantiles of the distribution), which have optimal dependence on the ambient dimension and which remain sharp under degeneracy. The following summarizes the main results.

\begin{theorem}[Main results, informal]\label{thm:main-informal}
       Let $(Z_1, \ldots, Z_p) \sim N(0,\Sigma)$ be a non-trivial, centered Gaussian vector in $\mathbb{R}^p$ for $p \ge 3$. Define the signed and unsigned maxima
    \(M = \max_{1 \le i \le p}Z_i\) and \(M^* = \max_{1 \le i \le p}|Z_i|\), and  put
    \(
        \sigma_{\max}^2 = \max_{1 \le i \le p}\mathbb{E}[Z_i^2].
    \) Then the law of $M$ has a density $f$ on $(0,\infty)$\footnote{By our convention, this does not restrict the law of $M$ on the negative line and allows a mass point at zero.} and for $t > 0$, \[f(t) \le \frac{4\log p}{t}.\]  This further implies:
    \begin{enumerate}[label = (\roman*)]

        \smallskip

        \item\label{item:variance-weak} $f(t) \lesssim (\sigma_\text{max} \vee \mu)^{-1}\log p$ holds for all $t$ exceeding the $\nicefrac{2}{3}$ quantile of $M$;

        \smallskip

        \item\label{item:unsigned-density} $M^*$ is continuously distributed with density $f^*$, and $f^*(t) \lesssim \sigma_{\max}^{-1}\log(2p)$ for all $t \in \mathbb{R}$;

        \item\label{item:variance-lb} if $\mu = \mathbb{E}[M]$ is positive, then
        \(
            \sqrt{\Var(M)} \gtrsim \mu / (\log p).
        \)

        \smallskip



    \end{enumerate}
\end{theorem}

The first stated bound imposes no condition on the covariance $\Sigma$ whatsoever; it is scale-free in this sense.  
Implication \ref{item:variance-weak} shows that the density is bounded at testing-relevant quantiles whenever $\max_i\sigma_i$ is bounded away from zero, regardless of how many entries of $Z$ have near-zero variance or how correlated they may be (Proposition \ref{prop:general-finite-bound}).
 For unsigned maxima, implication \ref{item:unsigned-density} extends this to the whole real line.
 Implication \ref{item:variance-lb} is a lower bound on the variance of the maximum in terms of its expectation (Proposition \ref{prop:variance-bound}). 

 The conclusions of Theorem \ref{thm:main-informal} sharpen when further information on $\Sigma$ is available, such as the weak variance decay considered by \citet{lopes2020bootstrapping}; more detailed versions of the result and further discussion, including a characterization of anti-concentration under weak variance decay and an investigation of the density at extremal quantiles, are contained in Section \ref{sec:main-technical}. 

By combining Theorem \ref{thm:main-informal} with established couplings (see e.g.~\cite{koike2021notes}), we recover validity of critical values based on Gaussian and bootstrap approximations of maxima of high-dimensional sums under degeneracy. Formal results of this type, and analogous statements for the multiplier bootstrap, are given in Section \ref{sec:applications}. 

\begin{corollary}[Approximation under degeneracy, informal]\label{cor:informal}
    Let $S_n = \frac{1}{\sqrt n}\sum_{i=1}^n X_i$ be an i.n.i.d.~sum of centered random vectors in $\mathbb{R}^p$ with \emph{unrestricted} variance $\Sigma$, such that $\|X_{ij}\|_{\psi_1} \lesssim 1$, with $n, p \ge 3$. Let $Z$ be a Gaussian vector with variance $\Sigma$, and let $t_q$ be the $q^{\mathrm{th}}$ quantile of $\max_{1 \le j \le p} Z_j$. Let $\sigma_{\max}^2$ be the largest diagonal entry of $\Sigma$. Then, with $\Delta_{n,p} = \log^7(p)/ n$,
    \[
        \sup_{q \ge \nicefrac{2}{3}}\left|\,\mathbb{P}\Big\{\max_{1 \le j \le p}(S_n)_j \le t_q\Big\} - \mathbb{P}\Big\{\max_{1 \le j \le p} Z_j \le t_q\Big\}\right| \;\lesssim\; \sigma_{\max}^{-2/3}\Delta_{n,p}^{1/6}
    \]
    For the unsigned maximum, the approximation holds uniformly over all quantiles:
    \[
        \sup_{t \in \mathbb{R}}\left|\,\mathbb{P}\Big\{\max_{1 \le j \le p}|(S_n)_j| \le t\Big\} - \mathbb{P}\Big\{\max_{1 \le j \le p}| Z_j| \le t \Big\}\right|
        \;\lesssim\;
        \sigma_{\max}^{-2/3}\Delta_{n,p}^{1/6}.
    \]
\end{corollary}

Corollary \ref{cor:informal} and its bootstrap analogs (formally stated as Corollaries \ref{cor:nonuniform-clt}-\ref{cor:uniform-bootstrap}) show that a single non-degenerate coordinate is sufficient for high-dimensional approximation of the upper quantiles of $\max_j (S_n)_j$, and for \emph{all} quantiles of $\max_j |(S_n)_j|$, in the regime $\log^7(p)/n \to 0$. Finally, in the standard setting where $\mathbb{E}[M] \gtrsim \sigma_{\max}\sqrt{\log p}$, approximation of the signed maximum holds at all quantiles, and the dimension dependence improves to $\log^5(p)/n \to 0$.

To motivate our results, we next describe three examples of modern statistical problems involving approximation by degenerate Gaussian vectors. 

\subsection*{Example 1: robustness of sparse correlation detection} Detecting sparse correlations\textemdash whether some subset $S \subset [p]$ of at most $s$ entries of the covariate vector $X \in \mathbb{R}^p$ carries a non-trivial linear association with an outcome $Y$\textemdash is a fundamental task in high-dimensional inference. \citet{fan2018discoveries} consider the null hypothesis of no association, which amounts to a moment restriction that for all $\beta \in \mathbb{R}^p$ with $\|\beta\|_2 = 1$ and $\|\beta\|_0 \le s$,
\(\mathbb{E}[\bk{\beta}{X_i} Y_i] = 0\).\footnote{Here $\|\beta\|_0$ denotes the number of non-zero coordinates of the vector $\beta$.}
They show that the null distribution of 
\begin{equation}\label{eq:max-sparse-corr}
    M_{s,n} \;\coloneqq \; \sup_{\substack{\|\beta\|_2 = 1 \\ \|\beta\|_0 \le s}} \frac{1}{n}\sum_{i=1}^n \bk{\beta}{X_i}\, Y_i,
\end{equation}
is approximated by $\sup_\beta \bk{\beta}{Z}$ for $Z \sim N(0, \Sigma_n)$ with $\Sigma_n$ the sample covariance of $(Y_i X_i)_{i=1}^n$. 
The resulting test was shown to be valid provided that the \emph{restricted eigenvalue}
\begin{equation}\label{eq:gwas-restricted-eig}
    \nu_s \;=\; \inf_{\substack{\|\beta\|_2 = 1 \\ \|\beta\|_0 \le s}}\, \mathbb{E}\langle \beta, X_i\rangle^2
\end{equation}
is bounded away from zero. In their framework, the condition provides a natural route to the anti-concentration needed for bootstrap validity, as it also underlies methods used to discover sparse correlations. However, it is challenging to verify: lower-bounding $\nu_s$ from data is itself NP-hard \citep{bandeira2013certifying,dobriban2016regularity}. It also matters in practice: in genetic data, meiosis produces blocks of highly correlated coordinates whose effective dimension is well below the block length \citep{international2005haplotype}, yielding $s$-sparse contrasts $\beta$ with $\mathbb{E}\langle \beta, X_i\rangle^2 \approx 0$. A similar phenomenon arises in gene-expression data \cite{wang2019precision}.
Corollary \ref{cor:informal}, combined with a standard discrete approximation, suggests that bootstrap critical values for $M_{s,n}$ are asymptotically valid at every level $\alpha < \nicefrac{1}{3}$ with no condition on $\nu_s$ whatsoever. 

\subsection*{Example 2: asymptotic confidence sequences with arbitrary spending} In asymptotic confidence-sequence constructions, sequential test statistics converge to Brownian paths evaluated against time-dependent boundaries, with the boundary encoding the desired allocation of type-I error over time \citep{waudby2024time}. Rather than restricting attention to analytically tractable boundaries, one can calibrate a boundary by simulating the limiting Gaussian process and estimating a quantile of
\[
M_{A,T,\psi}
    = \sup_{t\in T,\;a\in A}\psi(t)\langle a,B_t\rangle,
\]
where $T$ is a grid of monitoring times and $A$ indexes outcomes or contrasts \citep{gnettner2025new}. The pointwise variance is proportional to $t\psi(t)^2$, which, for common boundaries, approaches zero at early times or over an infinite horizon. After discretization, our results support simulation-based calibration whenever the largest pointwise variance is non-trivial, without requiring the minimum variance to be separated from zero. This suggests a route to multivariate confidence sequences and implicitly specified spending rules.

\subsection*{Example 3: confidence bands under latent structure} Many modern datasets concentrate on low-complexity subsets of a much higher-dimensional ambient space, and a growing literature studies estimators that can adapt to the latent geometry of the data \citep{kim2019uniform,cleanthous2020kernel,berenfeld2024estimating,green2021minimax}.
More broadly, modern statistical problems increasingly involve structured or non-standard data for which degeneracy of the approximating process is unavoidable \citep{singh2023kernel}.
Suppose that a distribution \(P\) on \(\mathbb{R}^D\) has heterogeneous, locally low-dimensional structure, and consider inference on the smoothed target \(p_h(x)=\mathbb{E}\widehat p_h(x)\), where
\[\widehat p_h(x)=\frac{1}{nh^D}\sum_{i=1}^nk\left(\frac{x-X_i}{h}\right).\]
The pointwise variance, analyzed by \citet{kim2019uniform}, satisfies
\[\operatorname{Var}{\widehat p_h(x)}\lesssim n^{-1}h^{-2D}P\{B(x,h)\}\asymp n^{-1}h^{-(2D-d(x))}\]
when 
\(P\{B(x,r)\}\asymp r^{d(x)}\). 
Thus, changes in local dimension or local mass can produce large variance ratios: under the ambient normalization, absolute fluctuations are larger at lower \(d(x)\), while the variance may be negligible in low-mass regions or away from the effective support.  
This creates a mismatch with confidence band constructions motivated by classical non-parametric estimation, where it is natural for pointwise variances to be standardized \citep{chernozhukov2014anti}. Standardization produces conditions such as $\sup_x|\hat\sigma(x)/\sigma(x) - 1| = o_p(1)$, which are prohibitive when $D$ is large and $\sigma(x)$ can be tiny. In contrast, Corollary \ref{cor:informal} can be applied to the discretized function class without rescaling, avoiding the need to learn $\sigma(x)$ uniformly at multiplicative scale.

\subsection{Related work}

Densities of maxima of Gaussian vectors are classical objects of study, and \citet{chernozhukov2014anti} demonstrate their central role in high-dimensional and non-parametric statistics. They showed that controlling the density of the Gaussian maximum\textemdash together with an appropriate high-dimensional coupling\textemdash suffices to establish distributional approximation in several problems of practical and theoretical interest. A general anti-concentration bound was obtained for centered vectors by \citet{chernozhukov2015comparison}, yielding a density bound of order $(\min_i \sigma_i)^{-1}\sqrt{\log p}$, which is proved to be essentially optimal when all marginal variances are of constant order. Further work extended this result to the non-centered case \cite{chernozhukov2016empirical,chernozhukov2017central} via Nazarov's inequality on the Gaussian surface area of polytopes \cite{nazarov2004maximal,klivans2008learning}.

Subsequent work has sought to relax the dependence on $\min_i \sigma_i$. \citet{lopes2020bootstrapping}, in establishing improved bootstrap approximations under weak variance decay, also obtain improved density bounds by separating the contributions of large and small variances. \citet{deng2020beyond}, who provide novel non-Gaussian bootstrap couplings, further improve anti-concentration estimates in certain regimes including for non-centered Gaussian vectors (see also \cite{chernozhukov2016empirical}). More recently, \citet{giessing2023anti} studies general order statistics and relates the problem to control of the variance of the maximum. With the exception of \citet{lopes2020bootstrapping}, who exploit structured variance decay and control of correlations, dependence on the minimum coordinate variance persists in these works.

This paper eliminates the dependence on $\min_i\sigma_i$ for control of the density of centered maxima, under appropriate conditions, and it provides anti-concentration bounds which remain valid under degeneracy. 
Our approach exploits the geometry of centered maxima\textemdash specifically, that facets of the max polytope which correspond to small variance coordinates lie far from the origin.

\subsection*{Organization} The remainder of the paper is organized as follows. Section \ref{sec:main-technical} contains the main technical density bounds for Gaussian maxima: a non-uniform bound that requires essentially no condition on the covariance, location-uniform bounds and variance control under mild conditions on $\mathbb{E}[\max_i Z_i]$, a characterization of the density under weak variance decay, and a bound at extremal quantiles. Section \ref{sec:applications} works out the implications for Gaussian and multiplier-bootstrap approximation of maxima of high-dimensional sums under weakened variance restrictions. Section \ref{sec:proofs} gives the proofs of the main density bounds; Appendix \ref{apx:proofs} collects auxiliary results, and Appendix \ref{sec:application-proofs} gives the proofs of the implications.


\section{Anti-concentration of degenerate maxima}\label{sec:main-technical}

We begin with a density bound for arbitrary finite-dimensional Gaussian vectors. The novelty of the bound is that it is non-uniform in the location $t$, but uniform in the covariance $\Sigma \succeq 0$. We subsequently show how to ``lift'' the inequality to bounds which are uniform over $t \in \mathbb{R}$, or over relevant quantiles.

\subsection{A non-uniform, scale-free density bound}\label{sec:nonuniform-bound}

\begin{proposition}\label{prop:general-finite-bound}Let $Z \sim N(0,\Sigma)$ be a centered Gaussian vector in $\mathbb{R}^p$ for $p \ge 3$, with $\sigma_i^2 = \bb{E}[Z_i^2]$, and write $M = \max_{1 \le i \le p} Z_i$. Define the effective dimension $p^*(t) \le p$ according to
\[p^*(t) = \#\Big\{1 \le i \le p : \sigma_i \ge t(2 \log p + 2\log\log p)^{-\frac 1 2}\Big\} \vee2.\] Then, for any $t > 0$,
\begin{equation}
    \lim_{\delta \downarrow 0} \frac{1}{\delta} \mathbb{P}\left\{ t \le M < t+ \delta \right\} \le 4 \min\left\{ \frac{\log p^*(t) + \nicefrac{1}{4}}{t}, \frac{ \log p}{t}\right\}.\label{eq:general-finite-bound}
\end{equation}
\end{proposition}

\begin{proof}[Proof sketch] Proposition \ref{prop:general-finite-bound} follows from modifying Nazarov's bound on the Gaussian surface area of a polytope \cite{nazarov2004maximal,klivans2008learning}. Its proof, given in Section \ref{sec:proofs}, crucially exploits the fact that $Z$ is centered. 
As in the standard reduction to Nazarov's bound, one begins by noting that $\mathbb{P}\{M \le t\}$ coincides with the standard Gaussian measure of the polytope $P_t=\{u \in \mathbb{R}^p: \max_i \, (\Sigma^{\nicefrac 1 2}u)_i \le t\}$. Since $Z$ is centered, the facet of $P_t$ corresponding to the inequality $\{(\Sigma^{\nicefrac 1 2}u)_i \le t\}$ is separated from the origin by a distance of precisely $t/\sigma_{i}$. Tracking this data within Nazarov's argument, one obtains
\begin{equation}\label{eq:master-sketch}
\lim_{\delta \downarrow 0} \frac{1}{\delta} \mathbb{P}\left\{ t \le M < t+ \delta \right\} \le \left(\frac{u + t}{u^2}\right) + \sum_{0<\sigma_i < u} \frac{1}{\sigma_i} \vp\left(\frac{t}{\sigma_i}\right),
\end{equation}
where $\vp(t)$ is the Gaussian density. This automatically limits the contribution of low-variance coordinates: for example, coordinates with $\sigma_i < t(2 \log p + 2\log\log p)^{-\frac 1 2}$ contribute at most $1$. Optimizing the threshold $u$ to handle the remaining $p^*(t)$ coordinates produces the bound \eqref{eq:general-finite-bound}.
\end{proof}
 
We emphasize that the bound $(4 \log p) / t$ in \eqref{eq:general-finite-bound} provides a fixed envelope that does not depend on the covariance $\Sigma$ at all.
This lack of dependence may seem unusual. For example, one could choose a large number $N \gg 1$ and apply the bound \eqref{eq:general-finite-bound} to $Z/N$, thus obtaining the same upper bound for the densities of both $Z$ and $Z/N$ at the point $t$. However, this ``free lunch'' is counteracted by the non-uniform nature of the bound: interesting quantiles of $\max_i Z_i/N$ will be very close to $0$, and the bound scales as $t^{-1}$, so the bound increases by the necessary factor $N$ at the corresponding \emph{quantiles} $t/N$ of the law of $Z/N$.

Still, for $t \ge 1$, we obtain an $O(\log p)$ anti-concentration bound with no restriction on $\Sigma$, and the bound improves for large values of $t$. This is already enough to cover many interesting quantiles of $\max_i Z_i / \sigma_{\max}$, as follows by comparison to the standard normal law, since quantiles of the maximum are larger than that of each coordinate. Thus, \eqref{eq:general-finite-bound} controls the density of $\max_i Z_i$ at all quantiles $q \ge \nicefrac{1}{2} + \delta$, up to a universal constant $C_\delta = O(\nicefrac{1}{\delta})$, provided $Z$ satisfies $\mu \vee \sigma_{\max} \ge 1$. 

\begin{corollary}\label{cor:upper-tail-bound}
    In the setting of Proposition \ref{prop:general-finite-bound}, let $q_{r}^M$ be the $r$-quantile of the law of $M$ for some $r > \nicefrac{1}{2}$, let $\mu = \mathbb{E}[M]$, and put $\sigma_{\max} = \max_i \sigma_i$. Then, for all $t \ge q_{r}^M$ and with $C_r = 1/(r - \nicefrac{1}{2})$, \[f(t) \lesssim \frac{C_r \log p}{\sigma_{\max} \vee \mu}.\]
\end{corollary}

Before proving Corollary \ref{cor:upper-tail-bound}, we state two convenient facts that will be used repeatedly (see e.g.~\citealp[Appendix A]{chatterjee2014superconcentration}). The first is the standard Gaussian maximal inequality,
\begin{equation}
    \label{eq:maximal} \mu \le \sigma_{\max}\sqrt{2\log p}.\tag{GM}
\end{equation}
The second is the Gaussian concentration inequality of \citet{borell1975brunn} and \citet{cirel1976norms}:
\begin{equation}
    \mathbb{P}\{M \le \mu - u\sigma_{\max}\} \;\le\;e^{-u^2/2}; \qquad \mathbb{P}\{M \ge \mu + u\sigma_{\max}\} \;\le\;e^{-u^2/2}.\tag{GC}\label{eq:borell-tis}
\end{equation}

\begin{proof}[Proof of Corollary \ref{cor:upper-tail-bound}]
By \eqref{eq:general-finite-bound}, it suffices to find $c_{1,r} \sigma_{\max}$ and $c_{2,r}$ such that $c_{1,r} \sigma_{\max}\vee c_{2,r}\mu \le q_{r}^M$. By monotonicity of the maximum, $q_{r}^M \ge \sigma_{\max}\Phi^{-1}(r)$. Thus, we may take $c_{1,r} = \Phi^{-1}(r) \gtrsim r - \nicefrac{1}{2}$.

It remains to work out $c_{2,r}$. Suppose first that $t \mu \le \sigma_{\max}$ for a positive constant $t$ to be specified in the proof. 
  Then $q_{r}^M  \ge \Phi^{-1}(r) t \mu$, and we may choose $c_{2,r} = \Phi^{-1}(r) t$.
  On the other hand, if $t \mu > \sigma_{\max}$, then  \eqref{eq:borell-tis} ensures that 
 \[\mu - \sqrt{2\log(\nicefrac{1}{r})} t \mu \le \mu -  \sigma_{\max}\sqrt{2\log(\nicefrac{1}{r})} \le q_{r}^M.\] Thus we may take $c_{2,r} = 1 - t\sqrt{2\log(\nicefrac{1}{r})}$ in this case. To cover both cases, choosing $t = \{\Phi^{-1}(r) +  \sqrt{2\log(\nicefrac{1}{r})}\}^{-1}$ implies that we may take 
 \[c_{2,r} \ge \Phi^{-1}(r) \Big/ \left\{\Phi^{-1}(r) +  \sqrt{2\log(\nicefrac{1}{r})}\right\} \asymp (r - \nicefrac{1}{2}). \qedhere\]
\end{proof}

\subsection{Uniform density bounds for unsigned maxima}

The density bound of Proposition \ref{prop:general-finite-bound} says, in effect, that the density of the Gaussian maximum $M$ on the positive line can only be large near zero: $f(t) \le (4\log p)/t$ for any $t > 0$.

On the other hand, for the unsigned maximum $M^* = \max_i |Z_i|$, the normalization of $\sigma_{\max} = 1$ automatically limits the extent to which probability can accumulate near zero: coupling $M^*$ to the magnitude of the largest-variance coordinate gives
\[\mathbb{P}\{M^* \le t\} \le \int_{-t}^t\vp(s)\,ds \le t\sqrt{2/\pi}.\] These two observations point towards the following result, which demonstrates that the density of  $M^*$ is uniformly bounded by $\{4\log(2p)\}/\sigma_{\max}$. 

\begin{proposition}\label{prop:unsigned-density-bound}
Let $Z \sim N(0,\Sigma)$ be a centered Gaussian vector in $\mathbb{R}^p$ with $p \ge 3$, let $\sigma_{\max}^2 = \max_{1 \le i \le p} \mathbb{E}[Z_i^2]$, and let $M^* = \max_{1 \le i \le p} |Z_i|$. Then, for any $t \ge 0$,
\begin{equation}
    \lim_{\delta \downarrow 0} \frac{1}{\delta} \mathbb{P}\left\{ t \le M^* < t+ \delta \right\} \le  \frac{4\log(2p)}{\sigma_{\max} \vee t}.\label{eq:unsigned-density-bound}
\end{equation}
\end{proposition}

\begin{proof}
    Note that we can represent $M^*$ as the maximum of $2p$ jointly Gaussian coordinates. Thus, the case $t \ge \sigma_{\max}$ is immediate by Proposition \ref{prop:general-finite-bound}. We now handle the case of $t < \sigma_{\max}$. 
    
     Without loss of generality, we rescale $Z$ and rearrange coordinates so that $\mathbb{E}[Z_1^2] = \sigma_{\max}^2 = 1$. We begin by defining coefficients $\beta_{1i} = \mathbb{E}[Z_1Z_i]$ and linear residuals $Z^{\perp}_i = Z_i - \beta_{1i}Z_1$ for $1 \le i \le p$ (so $|\beta_{1i}| \le 1$,  $Z^\perp$ is centered, and  $Z_1^\perp = 0$). By construction, $Z_1$ and $Z^\perp$ are independent, giving
    \[\mathbb{P}\{M^* \le t\} = \int_{-t}^t \vp(z_1)\mathbb{P}\left\{\max_{2 \le i \le p} |\beta_{1i}z_1 + Z_i^\perp| \le t
    \right\}\,dz_1.\]
    The key step is to note that the inner probability can be represented as a signed, centered maximum in dimension $2p-2$. Indeed, a simple rearrangement shows \[|\beta_{1i}z_1 + Z_i^\perp| \le t \iff \max\left\{\frac{Z_i^\perp}{(1 - \beta_{1i}z_1/t)}, \frac{-Z_i^\perp}{(1 + \beta_{1i}z_1/t)}\right\} \le t.\]  Removing the null event $\{|Z_1|=t\}$ and writing $z_1 = tu$ for $|u| < 1$ gives the representation
    \[\mathbb{P}\{M^* \le t\} =  \int_{-1}^1 t\vp(tu)\mathbb{P}\left\{\max_{2 \le i \le p}\max_{\tau \in \{\pm 1\}} \frac{\tau Z_i^\perp}{(1 -\tau\beta_{1i} u)}\le t
    \right\}\,du.\]
    Write $H_u(t)$ for the inner probability.  By Proposition \ref{prop:general-finite-bound}, $H_u(t)$ is weakly differentiable at any $t > 0$, with derivative $h_u(t)$ bounded by $4 \log(2p-2)/t$. Thus we may differentiate to obtain
    \[\frac{d}{dt}\mathbb{P}\{M^* \le t\} = \int_{-1}^1 (1-t^2u^2)\vp(tu)H_u(t)\,du + t\int_{-1}^1 \vp(tu)h_u(t)\,du.\]
    Since $t$ and $|u|$ are in $(0,1)$, the first integrand is at most $1/\sqrt{2\pi}$, so this term contributes at most $\sqrt{2/\pi}$. For the second term,  \eqref{eq:general-finite-bound} gives $t h_u(t) \le 4\log(2p-2)$ leading to a bound of $\{8\log(2p-2) + 2\}/\sqrt{2\pi}$. When $p \ge 3$, this is always at most $4\log(2p)$; reversing the normalization $\sigma_{\max}=1$ gives \eqref{eq:unsigned-density-bound}.
\end{proof}


\subsection{Bounds on the variance}\label{sec:variance-bounds}

Proposition \ref{prop:general-finite-bound} also implies bounds on the variance of the maximum which depend only upon its expectation and the dimension, eliminating further dependence on $\Sigma$.

\begin{proposition}\label{prop:variance-bound}
    Let $Z \sim N(0,\Sigma)$ be a centered Gaussian vector in $\mathbb{R}^p$ with $p \ge 3$, and let $M = \max_{1 \le i \le p} Z_i$, and let $\mu= \mathbb{E}[M] $ denote its expectation. Then, if $\mu > 0$, we have
\begin{equation}
    \sqrt{\Var(M)} \ge \frac{\mu}{15 \log p}.\label{eq:variance-bound}
\end{equation}
\end{proposition}

Proposition \ref{prop:variance-bound} complements \citet[Theorem 1.18]{ding2015multiple}, who show the bound $\sqrt{\Var(M)} \gtrsim 1/\mu$ when all coordinate variances are of equal order (see \citet[Theorem 3]{chernozhukov2015comparison} for a related result). Interestingly, here the role of $\mu$ is reversed: whereas in the homogeneous case, $\mu$ must be large in order for $M$ to be very concentrated, here we find that large $\mu$ \emph{limits} concentration of $M$: due to the density envelope of Proposition \ref{prop:general-finite-bound}, large spikes can only lie close to $0$, forcing a variance contribution of order $\mu^2$. The two bounds agree when $\mu \asymp \sqrt{\log p}$.

\begin{proof}[Proof of Proposition \ref{prop:variance-bound}]
 With our scale-free bound \eqref{eq:general-finite-bound} in hand, lower-bounding the variance becomes a straightforward variational problem that can be solved exactly.
\begin{lemma}\label{lem:bathtub}
Let $X$ be a random variable with a density $g$ on $(0, \infty)$, such that (i) for all $t > 0$, $g(t) \le B/t$, and (ii) $\mathbb{E}[X] > 0$. Recall the definition $\coth(x) = (e^{2x} + 1)/(e^{2x}-1)$. Then
\[\Var(X) \ge \mathbb{E}[X]^2\left\{\frac{1}{2B}\coth\left(\frac{1}{2B}\right) - 1 \right\} \ge \frac{\mathbb{E}[X]^2}{13 B^2},\]
where the first inequality is tight, and the last inequality holds whenever $B \ge 1$.
\end{lemma} 

Lemma \ref{lem:bathtub} follows from a short tilting argument (Lemma \ref{lem:bathtub-principle}), after ruling out mass on $(-\infty,0]$. One can show that the minimum variance is attained by a density of the form $g_*(t) = (B/t)\mathbbm{1}\{a \le t \le b\}$, and then optimize over $b > a > 0$ subject to $\int g_*(t)\,dt = 1$. The full proof of the lemma is postponed to Section \ref{sec:uniform-proofs}. Proposition \ref{prop:variance-bound} is then proved by applying Lemma \ref{lem:bathtub} with $X = M$ and $B = 4\log(p)$.
\end{proof}

\subsubsection*{Variance-density bounds} One may consider lifting Proposition \ref{prop:variance-bound} into a uniform density bound using a statement of the following form.
\begin{assumption}\label{assn:kappa} The law of 
$M = \max_{1 \le i \le p} Z_i$ is absolutely continuous, and its density satisfies $\sup_{t \in \mathbb{R}} f(t) \le \sqrt{\kappa / \Var(M)}$. 
\end{assumption}
A variance-density principle of this form was recently proposed \cite[Theorem 2]{giessing2023anti}. However, as stated, it requires additional
qualification for any (even dimension-dependent) $\kappa$: take
 $Z \sim N(0, \Sigma)$ with $\Sigma = \operatorname{diag}(\delta,1, \ldots, 1)$ in any dimension $p \ge 2$.
 As $\delta \downarrow 0$, both mean and variance of $\max_i Z_i$ converge to positive constants, while the maximum density diverges as $\delta^{-\frac 1 2}$.\footnote{By our Proposition \ref{prop:general-finite-bound}, such a density spike can only occur in a neighborhood of $0$.}



\begin{remark}[Unsigned maxima]\label{rmk:unsigned}For unsigned maxima, $M^* = \max_i|Z_i|$, one may take $Z \sim N(0,\Sigma)$ for
 $\Sigma = \operatorname{diag}\{4, 1/(\log p), \ldots, 1/(\log p)\}$ to obtain $\sup_t f^*(t) \asymp \log p$ with $\sqrt{\Var(M^*)}$ of constant order. Up to constants, this is the largest possible gap: by integration of its tails, the Gaussian concentration inequality \eqref{eq:borell-tis} forces $\sqrt{\Var(M^*)} \lesssim \sigma_{\max}$, while our density bound \eqref{eq:unsigned-density-bound} forces $\sup_{t} f^*(t) \lesssim \sigma_{\max}^{-1}\log(p)$, implying $\sup_{t} f^*(t) \lesssim \log(p)/\sqrt{\Var(M^*)}$. 
 \end{remark}

\subsection{Anti-concentration under degeneracy}\label{sec:extensions-uniform}
 Propositions \ref{prop:general-finite-bound} and \ref{prop:unsigned-density-bound} imply anti-concentration bounds which have the heuristic scale $(\log p)/\mu$.
The following Corollary \ref{cor:finite-window} differs from standard anti-concentration bounds in its dependence on the global size of $Z$ via $\mu$ and $\sigma_{\mathrm{max}}$, as opposed to local size via $\min_i \sigma_i$. This makes it applicable to degenerate vectors and processes. 

\begin{corollary}\label{cor:finite-window}
Let $Z \sim N(0,\Sigma)$ be a non-trivial, centered Gaussian vector in $\mathbb{R}^p$ with $p \ge 3$, let $M = \max_{1 \le i \le p} Z_i$, and write $\mu = \mathbb{E}[M]$ and $\sigma_{\mathrm{max}}^2 = \max_{1 \le i \le p}\mathbb{E}[Z_i^2]$. Suppose that $\mu > 0$. Then, for every $t \in \mathbb{R}$ and every $\varepsilon \ge \sigma_{\mathrm{max}}\exp\{-\mu^2/(8\sigma_{\mathrm{max}}^2)\}$,
\begin{equation}\label{eq:finite-window-eps}
\mathbb{P}\{t \le M \le t+\varepsilon\} \;\lesssim\; \frac{\varepsilon\log p}{\mu}.
\end{equation}
Moreover, writing $M^* = \max_{1 \le i \le p }|Z_i|$ and $\mu^* = \mathbb{E}[M^*]$, we have for all $0< \varepsilon < \mu^*/e$,
\begin{equation}\label{eq:finite-window-unsigned}
\mathbb{P}\{t \le M^* \le t+\varepsilon\} \lesssim \sqrt{\log(\nicefrac{\mu^*}{\varepsilon})}\left\{\frac{\varepsilon \log (2p)}{\mu^*}\right\}.
\end{equation}
\end{corollary}
\begin{remark}[Tightness]\label{rmk:tight} The dependence on $p$ in \eqref{eq:finite-window-eps} is optimal. For $Z\sim N(0,I_p)$ one has $\mu = \mathbb{E}[M] \asymp \sqrt{\log p}$ and $\sup_t f(t) \asymp \sqrt{\log p}$ \cite{chatterjee2014superconcentration}. Since the bound is of order $\log p/\mu \asymp \sqrt{\log p}$, and for such non-degenerate vectors the finite-window anti-concentration in \eqref{eq:finite-window-eps} is of the same order as $\sup_t f$, the bound is attained. Rescaling $Z$ by the factor $(\log p)^{-\frac 1 2}$ gives $\mathbb{E}[M] \asymp 1$ and $\sup_t f(t) \asymp \log p$, matching the resulting order $\log p$. These examples also show that $\eqref{eq:finite-window-unsigned}$ is sharp up to the logarithmic factor.
\end{remark}

\begin{remark}[Lower restriction on $\varepsilon$] 
 The anti-concentration bound \eqref{eq:finite-window-eps} for signed maxima $M$ requires a minimum window size of $\varepsilon \ge \sigma_{\mathrm{max}}\exp\{-\mu^2/(8\sigma_{\mathrm{max}}^2)\}$. By adding an independent coordinate with variance $\delta \downarrow 0$ to both constructions above, one observes that the lower restriction on $\varepsilon$ is necessary in general, as the infinitesimal window $(0,\sqrt{\delta})$ carries constant mass of order $2^{-p}$. This restriction is no longer needed for unsigned maxima $M^*$.
\end{remark}

\begin{remark}[Comparisons] When all marginal standard deviations are equal to $\sigma$, \citet{chernozhukov2015comparison} give an anti-concentration bound of order $\sigma^{-1}(1+\mu/\sigma)$, which is sharp in its dependence on $p$ for independent coordinates. In the regime $\mu \asymp \sigma_{\max}\sqrt{\log p}$, the bound \eqref{eq:finite-window-eps} is of order $\sqrt{\log p}/\sigma_{\max}$, matching this benchmark while allowing $\min_i\sigma_i/\sigma_{\max}$ to vanish. The improved anti-concentration bound of \citet{deng2020beyond} is of order $(\min_i \sigma_i)^{-1}$ in its most favorable cases (roughly, when $\sigma_i \gtrsim \sigma_p\sqrt{\log\{p-i\}}$ for each $1 \le i < p$). Relative to this rate, the coefficient in \eqref{eq:finite-window-eps} differs by the factor $(\min_i\sigma_i)\log p/\mu$. Thus, our bound is sharper when this ratio is small and can be sharper by an arbitrarily large factor when it tends to zero. Conversely, a loss of order $\sqrt{\log p}$ occurs when $\mu \asymp (\min_i\sigma_i)\sqrt{\log p}$. A major difference is that \citet{deng2020beyond} also consider non-centered vectors (see also \citet{chernozhukov2016empirical}). Another difference is the implicit condition $\mu \gg \sigma_{\max}$, which is unavoidable for degenerate vectors due to possible mass at $0$, and which is eliminated by \eqref{eq:finite-window-unsigned} for unsigned maxima.
\end{remark}
\begin{proof}
Fix $t \in \mathbb{R}$ and $\varepsilon > 0$. For any $a \in [0,\mu)$, splitting the window $[t,t+\varepsilon]$ at $a$ gives
\begin{equation}\label{eq:finite-window-split}
\mathbb{P}\{t \le M \le t+\varepsilon\} \;\le\; \mathbb{P}\{M \le a\} + \int_{t \mathrel{\vee} a}^{t + \varepsilon} f(s)\,ds \le \mathbb{P}\{M \le a\} + \frac{4\varepsilon\log\,p}{a},
\end{equation}
by our density bound \eqref{eq:general-finite-bound}. For \eqref{eq:finite-window-eps}, take $a = \mu/2$. The first term is bounded by \eqref{eq:borell-tis} as
\begin{equation*}
\mathbb{P}\{M \le a\} \le e^{-(\mu-a)^2/2\sigma_{\max}^2}= e^{-\mu^2/{8\sigma_{\mathrm{max}}^2}} \le \varepsilon/\sigma_{\max},
\end{equation*}
where the last step follows by our lower restriction on $\varepsilon$. Consolidating the two terms in \eqref{eq:finite-window-split}, and using \eqref{eq:maximal} to deduce that $\varepsilon/\sigma_{\max} \le \sqrt{2\log p}(\varepsilon / \mu)$, yields \eqref{eq:finite-window-eps}.

For \eqref{eq:finite-window-unsigned}, note that $M^*$ is the \emph{signed} maximum of a Gaussian vector in dimension $2p$, so that \eqref{eq:finite-window-split} holds with $M$ and $\log p$ replaced by $M^*$ and $\log(2p)$. Now let $\varepsilon \in (0,\mu^*/e)$ be given. We consider cases. In the first case, suppose that $\mu^* \le 8\sigma_{\max}\sqrt{\log(\nicefrac{\mu^*}{\varepsilon})}$. In this case, we use \eqref{eq:unsigned-density-bound} directly:
\[\int_{t}^{t + \varepsilon}f^*(s)\,ds \le \frac{4\varepsilon\log(2p)}{\sigma_{\max}} \le \varepsilon\sqrt{\log(\nicefrac{\mu^*}{\varepsilon})}\left\{\frac{32\log(2p)}{\mu^*}\right\}.\]
On the other hand, if $\mu^* > 8\sigma_{\max}\sqrt{\log(\nicefrac{\mu^*}{\varepsilon})}$ then choose $a = \mu^*/2$; \eqref{eq:borell-tis} then gives $\mathbb{P}\{M^* \le a\} \le \varepsilon/\mu^*$. Using this bound in \eqref{eq:finite-window-split} gives
\[\mathbb{P}\{t \le M^* \le t+\varepsilon\} \le \frac{\varepsilon}{\mu^*} + \frac{8\varepsilon\log\,(2p)}{\mu^*}.\]
Consolidating the two bounds gives \eqref{eq:finite-window-unsigned}.
\end{proof}

\subsection{Weak variance decay}\label{sec:weak-var-decay}
We now consider the behavior of our bounds under further structural assumptions, particularly the weak variance decay considered by \citet{lopes2020bootstrapping}.
Without loss of generality, suppose that the standard deviations $\{\sigma_q\}_{1 \le q \le p}$ are given in decreasing order. 
The bounds of Section \ref{sec:nonuniform-bound} admit the following extension, implying a dimension-independent bound when $\sigma_q = o(1/\sqrt{\log q})$.

\begin{proposition}\label{prop:weak-var-decay}
    For $u > 0$, let $S_u = \{q: \sigma_q \ge u\}$ be the set of coordinate indices whose standard deviations exceed $u$ and $S_u^c$ be the complementary set. We then have
    \begin{equation}\label{eq:weak}
        f(t) \lesssim \inf_{u > 0}\left\{\frac{\sqrt{\log (\#S_u \vee 2)}}{u}  + \sum_{q \in S_u^c} \frac{1}{\sigma_q} \exp\left(\frac{-t^2}{2\sigma_q^2}\right)\right\}.
    \end{equation}
\end{proposition}
\begin{proof}
See Section \ref{sec:proof-prop1}. \let\qed\relax
\end{proof}
Proposition \ref{prop:weak-var-decay} recovers the following sharp boundary in the behavior of the density under weak variance decay. The result is based upon convergence of the sum appearing in \eqref{eq:weak}.
\begin{example}[Boundedness under weak variance decay]
    Suppose that for some $\alpha > 0$, the vector $Z \in \mathbb{R}^p$ satisfies (after a possible rearrangement of the coordinates):
    \begin{equation}
        \sigma_k \le \frac{t}{\sqrt{2\alpha\log k}} \label{eq:very-weak-variance-decay}.
    \end{equation}
     If $\alpha > 1$, then $f(t)$ is bounded by an envelope which depends only upon $\alpha$ (and not $p$). 
    
     On the other hand, for any $p$ and any $t > 0$, we may take $Z$ to be a Gaussian with variance $t^2(2\log p)^{-1}I_p$. This construction satisfies \eqref{eq:very-weak-variance-decay} with the precise constant $\alpha = 1$, and it has $f(t) \asymp \sqrt{\log p}/t$ for each $p$, so that the density is unbounded as $p$ grows. In particular, this construction has $\mu \sim t$ as $p \uparrow \infty$, and $M - \mu = O_p(t/\log p)$ \cite[Appendix A]{chatterjee2014superconcentration}.

     Thus, \eqref{eq:weak} captures the threshold for uniform boundedness of the density under this variance profile, up to the sharp constant $\alpha = 1$ in the regime $\sigma_k \asymp (2\alpha\log k)^{-\frac 1 2}$.
\end{example}

\subsection{Density at extremal quantiles}\label{sec:extremal-quantiles}

Arguments based on concavity of the Gaussian law can provide a more complete picture of the density at extremal points. As usual, we let $f$ denote the density of $M = \max_i Z_i$ on $(0,\infty)$ and we let $F(t) = \mathbb{P}\{M \le t\}$ be its cumulative distribution function. 

\begin{lemma}\label{prop:extremal-quantile}
    Let $Z \sim N(0,\Sigma)$ be a centered Gaussian vector in $\mathbb{R}^p$, let $M = \max_{1 \le i \le p} Z_i$ with density $f$ on $(0,\infty)$, and let $m$ be a median of $M$. Then, whenever $m > 0$, $f$ is non-increasing on $[m,\infty)$, and for all $h \ge 0$,
    \begin{equation}\label{eq:extremal-density-bound}
        f(m + h\sigma_{\mathrm{max}}) \le \sqrt{2\pi}\, f(m)\, \vp(h) = f(m)\, e^{-h^2/2},
    \end{equation}
    where $\vp$ is the standard Gaussian density. Thus, the non-uniform bound \eqref{eq:general-finite-bound} gives the envelope
    \[
     \frac{4\log p}{m}\, e^{-h^2/2} \quad (h \ge 0).
    \]
\end{lemma}

Lemma \ref{prop:extremal-quantile} provides a differential analog of the Gaussian concentration inequality, which states that $1 - F(m + h\sigma_{\mathrm{max}}) \le 1 - \Phi(h)$. It follows from the celebrated inequality of \citet{ehrhard1983symetrisation}, which implies the following.
\begin{lemma}[Ehrhard concavity]\label{thm:ehrard}
    The map $G(t) \!=\! \Phi^{-1}\{F(t)\}$ is concave on the support of $F$.
\end{lemma}
\citet{bobkov2008note} showed, moreover, that $G(t)$ is concave {if and only if} $F$ is the cumulative distribution function of the supremum of a separable Gaussian process.

\begin{proof}[Proof of Lemma \ref{prop:extremal-quantile}]
Let $G(s) = \Phi^{-1}\{F(s)\}$, so that $F(s) = \Phi(G(s))$. By Lemma \ref{thm:ehrard}, $G$ is concave, hence differentiable almost everywhere with $G'$ non-increasing; at these points,
\(
f(s) = G'(s)\,\vp(G(s)).
\)
We restrict to points of differentiability and extend to all $h \ge 0$ by continuity. This immediately gives
\begin{equation}\label{eq:G-prime-at-median}
G'(m) = f(m)/\vp(0) = \sqrt{2\pi}\,f(m).
\end{equation}
Gaussian concentration gives the slope-one lower bound
\[
G(m + h\sigma_{\mathrm{max}}) \;\ge\; h, \qquad h \ge 0.
\]
Combining this with $G' \downarrow$ and $G(m+h\sigma_{\mathrm{max}}) \ge h \ge 0$, which implies $\vp(G(m + h\sigma_{\mathrm{max}})) \le \vp(h)$, we obtain
\[
f(m + h\sigma_{\mathrm{max}})
\;=\; G'(m + h\sigma_{\mathrm{max}})\,\vp(G(m + h\sigma_{\mathrm{max}}))
\;\le\; G'(m)\,\vp(h)
\;=\; \sqrt{2\pi}\, f(m)\,\vp(h),
\]
which is \eqref{eq:extremal-density-bound}. For the envelope, note that if $m > 0$ the non-uniform bound \eqref{eq:general-finite-bound} applies at $t = m$ to give $f(m) \le 4\log p/m$; combined with the monotonicity of $f$ on $[m,\infty)$ established above, this yields $\sup_{t \ge m} f(t) = f(m) \le 4\log p/m$ and the displayed Gaussian envelope.
\end{proof}


\section{Some implications}\label{sec:applications}

In this section, we work out the implications of the bounds of Section \ref{sec:main-technical} for Gaussian and bootstrap approximation of maxima of high-dimensional sums. The novelty is that these approximations place weaker restrictions on the pointwise variances, and thus they remain applicable under degeneracy.

We rely on the high-dimensional couplings of \citet{koike2021notes}, which build on the seminal Gaussian and bootstrap approximation bounds of \citet{chernozhukov2017central} using a randomized Lindeberg construction. The coupling is the substantive ingredient; our density bounds enter only as a replacement for Koike's application of Nazarov's surface area inequality, and the only work is subsequent optimization of the bound. In other words, the corollaries below follow by combining the two results with little additional effort. We focus on this particular coupling because its proof needs only anti-concentration of \emph{centered} Gaussian maxima\textemdash the regime in which our results apply. Subsequent improvements to the high-dimensional CLT \citep[e.g.,][]{chernozhukov2022improved} achieve tighter dimension dependence but rely internally on anti-concentration of non-centered Gaussian vectors, which we do not address.

Throughout this section, $X_1, \ldots, X_n$ are independent, centered random vectors in $\mathbb{R}^p$ whose entries satisfy the sub-exponential bound $\|X_{ij}\|_{\psi_1} \le K$ for some constant $K \ge 1$. We write $S_n = \frac{1}{\sqrt n}\sum_{i=1}^n X_i$ for the normalized sum and $\Sigma = \mathbb{E}[S_n S_n^\top]$ for its covariance, with largest marginal variance
\begin{equation}\label{eq:nontrivial-var}
\sigma_{\mathrm{max}}^2 = \max_{1 \le j \le p} \Sigma_{jj} = \max_{1 \le j \le p} \frac{1}{n}\sum_{i=1}^n \mathbb{E}[X_{ij}^2].
\end{equation}
We let $Z \sim N(0,\Sigma)$ denote the Gaussian analog of $S_n$, and $t_q^\Sigma$ the $q^{\mathrm{th}}$ quantile of $\max_{1 \le j \le p} Z_j$. For the bootstrap results we additionally fix multipliers $(w_i)_{i=1}^n$, i.i.d.\ and independent of $X = (X_1, \ldots, X_n)$, with $\mathbb{E}[w_i] = 0$, $\mathbb{E}[w_i^2] = 1$, and $|w_i| \le b$ almost surely for some constant $b \ge 1$, and form the multiplier (wild) bootstrap statistic
\[
S_n^{\mathrm{WB}} = \frac{1}{\sqrt n}\sum_{i=1}^n w_i (X_i - \bar X_n), \qquad \bar X_n = \frac{1}{n}\sum_{i=1}^n X_i.
\]

We additionally impose the following mild large-sample condition, which ensures that the lower-order remainder term in the coupling of \citet{koike2021notes} is dominated.

\begin{assumption}\label{assn:sampling}
The quantities above satisfy
\[
b\,K^2(\log p)^{1/2}(\log n)^2 \;\le\; \sqrt n,
\]
with $b = 1$ understood for results involving no multipliers.
\end{assumption}

Assumption \ref{assn:sampling} can be dropped whenever $K, b = O(1)$, as a non-trivial rate already requires $\log^7 p \lesssim n$, and then $(\log p)^{1/2}(\log n)^2 = o(\sqrt n)$. It plays no role beyond ensuring the second remainder term of the coupling is of lower order; see the proof of Corollaries \ref{cor:nonuniform-clt}--\ref{cor:uniform-bootstrap} in Appendix \ref{sec:application-proofs}.

\subsection{Gaussian approximation}\label{sec:app-clt}

The utility of the approximations below is their immediate ability to handle vectors $X_i$, which may have many coordinates with very small variance. In fact, the main requirement is that $\sigma_{\mathrm{max}}^2 = \max_{j} \frac{1}{n}\sum_{i=1}^n \mathbb{E}X_{ij}^2$ is not too small, which is enforced by a simple, global rescaling; note that \eqref{eq:nonuniform-clt} is scale-invariant.

\begin{corollary}[CLT, signed maxima]\label{cor:nonuniform-clt}
    Under Assumption~\ref{assn:sampling}, we have
    \begin{equation}\label{eq:nonuniform-clt}
    \sup_{q \ge \frac{2}{3}} \left| \mathbb{P}\left\{\max_{1 \le j \le p} Z_j \le t_{q}^\Sigma\right\} - \mathbb{P}\left\{\max_{1 \le j \le p} \frac{1}{\sqrt n} \sum_{i=1}^{n}X_{ij}\le t_{q}^\Sigma \right\}\right| \lesssim {\sigma_{\mathrm{max}}^{-2/3}}\left(\frac{K^4\log^7p}{n}\right)^{\frac 1 6}
    .
    \end{equation}
\end{corollary}

For unsigned maxima, the same rate holds uniformly over all quantiles.

\begin{corollary}[CLT, unsigned maxima]\label{cor:uniform-clt-unsigned}
    Under Assumption~\ref{assn:sampling},
    \begin{equation}\label{eq:uniform-clt-unsigned}
        \sup_{t \in \mathbb{R}} \left| \mathbb{P}\Big\{\max_{1 \le j \le p}|(S_n)_j| \le t\Big\} - \mathbb{P}\Big\{\max_{1 \le j \le p}|Z_j| \le t\Big\}\right|
        \;\lesssim\;
        \sigma_{\mathrm{max}}^{-2/3}\left(\frac{K^4\log^7(2p)}{n}\right)^{1/6}.
    \end{equation}
\end{corollary}

\subsection{Bootstrap approximation}\label{sec:app-bootstrap}

A parallel result holds for the multiplier (or wild) bootstrap statistic $S_n^{\mathrm{WB}}$ introduced above, analyzed by \citet{koike2021notes}, whose conditional law given the data $X = (X_1, \ldots, X_n)$ approximates that of $\max_{1 \le j \le p} Z_j$.

\begin{corollary}[Bootstrap, signed maxima]\label{cor:nonuniform-bootstrap}
    Under Assumption~\ref{assn:sampling},
    \begin{equation}\label{eq:nonuniform-bootstrap}
        \mathbb{E}\!\left[\sup_{q \ge \frac{2}{3}}\, \left|\mathbb{P}\!\left\{\max_{1 \le j \le p}\, (S_n^{\mathrm{WB}})_j \le t_q^\Sigma\,\big|\, X\right\} - \mathbb{P}\!\left\{\max_{1 \le j \le p} Z_j \le t_q^\Sigma\right\}\right|\right]
        \lesssim \sigma_{\mathrm{max}}^{-2/3}\left(\frac{b^2 K^4\log^7 p}{n}\right)^{\frac{1}{6}}.
    \end{equation}
\end{corollary}

\begin{corollary}[Bootstrap, unsigned maxima]\label{cor:uniform-bootstrap-unsigned}
    Under Assumption~\ref{assn:sampling},
    \begin{equation}\label{eq:uniform-bootstrap-unsigned}
        \mathbb{E}\sup_{t \in \mathbb{R}} \left| \mathbb{P}\Big\{\max_{1 \le j \le p}|(S_n^{\mathrm{WB}})_j| \le t \,\big|\, X\Big\} - \mathbb{P}\Big\{\max_{1 \le j \le p}|Z_j| \le t\Big\}\right|
        \;\lesssim\;
        \sigma_{\mathrm{max}}^{-2/3}\left(\frac{b^2K^4\log^7(2p)}{n}\right)^{1/6}.
    \end{equation}
\end{corollary}

\subsection{Uniform extension for signed maxima}\label{sec:app-uniform}

The preceding signed results concern upper quantiles. Under a complexity condition, namely that $\mu \gg \sigma_{\max}$, they extend to all quantiles and yield a high-dimensional CLT in Kolmogorov distance for possibly degenerate, centered maxima.

\begin{corollary}[Uniform CLT, signed maxima]\label{cor:uniform-clt}
    Put $\mu = \mathbb{E}[\max_{1 \le j \le p} Z_j] > 0$. Under Assumption~\ref{assn:sampling}, we have
    \begin{equation}\label{eq:uniform-clt}
        \sup_{t \in \mathbb{R}} \left| \mathbb{P}\left\{\max_{1 \le j \le p} Z_j \le t\right\} - \mathbb{P}\left\{\max_{1 \le j \le p} \frac{1}{\sqrt n} \sum_{i=1}^{n}X_{ij}\le t \right\}\right|  \lesssim e^{\frac{-\mu^2}{8\sigma_{\mathrm{max}}^2}} + \mu^{-2/3}\left(\frac{K^4\log^7p}{n\,}\right)^{\frac 1 6}.
    \end{equation}
\end{corollary}

\begin{corollary}[Uniform bootstrap, signed maxima]\label{cor:uniform-bootstrap}
    With $\mu$ as in Corollary~\ref{cor:uniform-clt}, under Assumption~\ref{assn:sampling},
    \begin{equation}\label{eq:uniform-bootstrap}
        \mathbb{E} \sup_{t \in \mathbb{R}}\, \left|\mathbb{P}\!\left\{\max_{1 \le j \le p}\, (S_n^{\mathrm{WB}})_j \le t \,\big|\, X\right\} - \mathbb{P}\!\left\{\max_{1 \le j \le p} Z_j \le t\right\}\right|
        \lesssim e^{\frac{-\mu^2}{8\sigma_{\mathrm{max}}^2}}+ \mu^{-2/3}\left(\frac{b^2 K^4\log^7 p}{n}\right)^{\frac{1}{6}}.
    \end{equation}
\end{corollary}

Corollaries \ref{cor:nonuniform-clt}-\ref{cor:uniform-bootstrap} require no lower bound on the minimum variance $\min_j \sigma_j$, in contrast to similar results which depend explicitly on $(\min_j \sigma_j)^{-1}$ \citep{chernozhukov2014anti,chernozhukov2017central,koike2021notes}. Moreover, under the complexity condition $\mu \gg \sigma_{\max}\sqrt{\log\,p}$ (which holds for an equicorrelated field), the approximation error vanishes in any triangular array where $(\log^5 p)/n \to 0$.

\section{Proofs of main results}\label{sec:proofs}

\subsection{The master bound}\label{sec:master-bound}

We begin by stating a more general result, from which the bounds of Sections \ref{sec:nonuniform-bound}  and \ref{sec:weak-var-decay} follow quite easily. Its proof will be given in Section \ref{sec:prop1-proofs}.

\begin{proposition} \label{prop:master-bound}
  Let $Z$ be a centered Gaussian random vector in $\mathbb{R}^p$ with covariance $\Sigma$.
  Let $\vp$ be the standard normal density.
  Then, for any $u,t > 0$, it holds that
      \begin{equation}\label{eq:master-bound}
         \lim_{\delta \downarrow 0} \frac{1}{\delta}\left( \bb{P}\left\{\max_{1 \le i\le p}Z_i \le t + \delta\right\}
- \bb{P}\left\{\max_{1 \le i\le p}Z_i \le t \right\} \right) \le \left(\frac{u + t}{u^2}\right) + \sum_{0<\sigma_i < u} \frac{1}{\sigma_i} \vp\left(\frac{t}{\sigma_i}\right).
    \end{equation}
In particular, $M = \max_{1 \le i \le p} Z_i$ admits a density on $(0,\infty)$.
\end{proposition}

\subsubsection*{Proof of Proposition \ref{prop:general-finite-bound} (scale-free bound)}\label{sec:proof-prop1}

We begin by explaining how Proposition \ref{prop:master-bound} may be used to recover the main results. Note that the function $x e^{-x^2}$ is decreasing for $x \ge 1/\sqrt{2}$ as its derivative is $(1-2x^2)e^{-x^2}$. Put $x = x(\sigma) = t/(\sigma\sqrt{2})$, which is a decreasing function of $\sigma$. It follows that if $\sigma \le t$, so that $x (\sigma) \ge 1/\sqrt{2}$, then
\begin{align*}
   \frac{1}{\sigma} \exp\left(\frac{-t^2}{2\sigma^2}\right) = \sqrt{2} t^{-1}xe^{-x^2}
\end{align*}
is decreasing in $x$, hence the left-hand side is an increasing function of $\sigma$. Thus, whenever $u \le t$ the bound \eqref{eq:master-bound} simplifies to
\begin{align*}
         \lim_{\delta \downarrow 0} \frac{1}{\delta}\left( \bb{P}\left\{\max_{1 \le i\le p}Z_i \le t + \delta\right\}
- \bb{P}\left\{\max_{1 \le i\le p}Z_i \le t \right\} \right)
&\le \left(\frac{u + t}{u^2}\right) + \frac{p}{\sqrt{2\pi}}\frac{1}{u} \exp\left(\frac{-t^2}{2u^2}\right) \\
& = \frac{2\log p}{t} + \left(\sqrt{2} + \frac{1}{\sqrt \pi}\right)\frac{\sqrt{\log p}}{t} \le \frac{4\log p}{t},
    \end{align*}
where the last line follows from taking $u = (2\log p)^{-\frac{1}{2}}t$ and simplifying. This proves one of the two claimed bounds.

Next, put
\(
u' = t(2 \log p + 2\log\log p)^{-\frac 1 2}
\)
and note that for
\(
\sigma \le u',
\)
we have
\(
\sigma^{-1}\vp(t/\sigma) \le (\sqrt{\pi}\,p\,t)^{-1}.
\)
Thus, our upper bound \eqref{eq:master-bound} may be analyzed as
\begin{align}
  \left(\frac{u + t}{u^2}\right) + \sum_{\sigma_i < u} \frac{1}{\sigma_i} \vp\left(\frac{t}{\sigma_i}\right)
  &= \left(\frac{u + t}{u^2}\right) + \sum_{u > \sigma_i > u'} \frac{1}{\sigma_i} \vp\left(\frac{t}{\sigma_i}\right) + \sum_{u' \ge \sigma_i } \frac{1}{\sigma_i} \vp\left(\frac{t}{\sigma_i}\right) \\
  &\le \left(\frac{u + t}{u^2}\right) + \sum_{u > \sigma_i > u'} \frac{1}{\sigma_i} \vp\left(\frac{t}{\sigma_i}\right) + \frac{1}{\sqrt{\pi}t}. \nonumber
\end{align}
Note that optimizing the bound given by the first two summands is analogous to the computation performed in the preceding paragraph, where we now take $u = \{2\log p^*(t)\}^{-\frac{1}{2}}t$. Here, $p^*(t)$ is the number of $\sigma_i$ that exceed $u'$. This proves the second claim of Proposition \ref{prop:general-finite-bound}.

\subsubsection*{Proof of Proposition \ref{prop:weak-var-decay} (weak variance decay)}\label{sec:proof-prop3}

We adapt the argument used to prove Proposition \ref{prop:general-finite-bound}. Fix the outer threshold $u > 0$ appearing in \eqref{eq:weak} and set
\(
k := \#S_u \vee 2= \#\{i : \sigma_i \ge u\} \vee2,
\)
so that every coordinate of $S_u$ has $\sigma_i \ge u$. The master bound (Proposition \ref{prop:master-bound}) holds at an arbitrary inner threshold $v > 0$,
\begin{equation}\label{eq:prop3-master}
f(t) \le \frac{v + t}{v^2} + \sum_{\sigma_i < v} \frac{1}{\sigma_i}\,\vp\!\left(\frac{t}{\sigma_i}\right);
\end{equation}
we bound the contribution of the first term and the summands $\sigma_i \ge u$ in two ways.

First, take $v = u_* := t/\sqrt{2\log k}$. The first term of \eqref{eq:prop3-master} satisfies $(u_* + t)/u_*^2 \le 2t/u_*^2 = 4 \log k / t$, since $u_* \le t$ for $k \ge 2$. We split the remaining sum at $u$, the second sum below being empty when $u \ge u_*$:
\begin{equation}\label{eq:prop3-split}
\sum_{\sigma_i < u_*} \frac{1}{\sigma_i}\,\vp\!\left(\frac{t}{\sigma_i}\right)
\;\le\; \sum_{u \le \sigma_i < u_*} \frac{1}{\sigma_i}\,\vp\!\left(\frac{t}{\sigma_i}\right)
\;+\; \sum_{\sigma_i < u} \frac{1}{\sigma_i}\,\vp\!\left(\frac{t}{\sigma_i}\right).
\end{equation}
The function $\sigma \mapsto \sigma^{-1}\vp(t/\sigma)$ is increasing on $(0,t]$, so each of the at most $k$ summands in the first sum is bounded by $u^{-1}\vp(t/u_*) = u^{-1}/(\sqrt{2\pi}\,k)$ by the definition of $u_*$; hence that sum is at most $1/(\sqrt{2\pi}\,u)$. Thus the contribution is at most $4\log k / t + 1/(\sqrt{2\pi}\,u)$.

Second, we may optimize the inner threshold in $t$. If $t \le u\sqrt{2\log k}$, take $v = u$: no coordinate of $S_u$ enters the sum in \eqref{eq:prop3-master}, and the first term obeys $(u + t)/u^2 \le (1 + \sqrt{2\log k})/u \lesssim \sqrt{\log k}/u$, using $k \ge 2$. If instead $t > u\sqrt{2\log k}$, take $v = u_* > u$; then $(u_* + t)/u_*^2 \le 4\log k / t < 2\sqrt{2}\,\sqrt{\log k}/u$, while the at most $k$ leftover coordinates in $[u, u_*)$ contribute, exactly as above, at most $1/(\sqrt{2\pi}\,u) \le \sqrt{\log k}/u$. Either way the contribution is $\lesssim \sqrt{\log k}/u$.

Taking the smaller of the two, and recalling $k = \#S_u \vee 2$, we obtain \eqref{eq:weak} up to absolute constants. Taking the infimum over $u > 0$ completes the proof.
\qed

\subsection{Proofs from Section \ref{sec:variance-bounds} (variance bounds).}\label{sec:uniform-proofs}

In this subsection we prove Lemma \ref{lem:bathtub}, the variational bound underlying Proposition \ref{prop:variance-bound}. Lower-bounding $\Var(M)$ amounts to a variational problem over densities under the pointwise restriction $f(t) \le B/t$ supplied by Proposition \ref{prop:general-finite-bound} (with $B = 4\log p$), which is solved exactly by the tilting argument stated next.

\begin{lemma}[{Bathtub principle, \citealp[cf.][Theorem 1.14]{lieb2001analysis}}]\label{lem:bathtub-principle}
Suppose we are given a measurable upper bound $u: [0,\infty) \to [0,\infty]$ and measurable cost functions $c, w_1, \ldots, w_k : [0,\infty) \to \mathbb{R}$. Given multipliers $\lambda_1, \ldots, \lambda_k \in \mathbb{R}$, define the tilted objective $g_\lambda(x) := c(x) - \sum_{m=1}^k \lambda_m w_m(x)$. Put
\[
    h_*(t) \coloneqq 
    \begin{cases} 
        u(t) & g_\lambda(t) < 0, \\
         0 & g_\lambda(t) > 0.
    \end{cases}
\]
Then $h_*(t)$ solves the optimization
\begin{equation*}
\begin{aligned}
\mathrm{minimize} \quad & \int_0^\infty c(t) h(t)\,dt \\
\mathrm{s.t.} \quad & 0 \le h(x) \le u(x), & x > 0 \\
& \int_0^\infty w_m(x)h(x)\,dx = \int_0^\infty w_m(x) h_*(x)\,dx,  &  1 \le m \le k
\end{aligned}
\end{equation*}
over measurable functions $h(t)$. 
\end{lemma}
\begin{proof}
Let $h(t)$ be any competitor. By construction, $g_\lambda(t)\,\{h(t) - h_*(t)\} \ge 0$: whenever $g_\lambda(x) < 0$ we have $h_*(x) = u(x) \ge h(x)$, and whenever $g_\lambda(x) > 0$ we have $h_*(x) = 0 \le h(x)$. Integrating and using that $\int_0^\infty w_m(t)\{h(t) - h_*(t)\}\,dt = 0$ for each $m$, we find
\begin{align*}
    0 &\le \int_0^\infty g_\lambda(t)\,\{h(t) - h_*(t)\}\,dt \\
    &= \int_0^\infty c(t)\,\{h(t) - h_*(t)\}\,dt - \sum_{m=1}^k \lambda_m \int_0^\infty w_m(t) \{h(t) - h_*(t)\}\,dt \\
    &= \int_0^\infty c(t)\,\{h(t) - h_*(t)\}\,dt. \qedhere
\end{align*}
\end{proof}

\subsubsection*{Proof of Lemma \ref{lem:bathtub}}
 Since any feasible distribution can be transported to a feasible distribution on the nonnegative half-line while reducing the cost $\mathbb E X^2$ and preserving the mean, via $X \mapsto X_+\,\mathbb E[X]/\mathbb E[X_+]$ with $X_+ = X\mathbf 1\{X\ge 0\}$, it suffices to consider the problem over distributions on the nonnegative half-line carrying mass $q$ on $(0,\infty)$ and an atom of mass $1-q$ at the origin, which affects neither the mean nor the second moment.

The extremal distribution must place its mass as close as possible to the mean while respecting the upper bound \(f(t)\le B/t\). We apply Lemma \ref{lem:bathtub-principle} with cost $c(t)=t^2$, upper bound $u(t)=B/t$, and the two constraints $w_1\equiv 1$ (mass $q$) and $w_2(t)=t$ (mean $A$). The tilt $g_\lambda(t)=t^2-\lambda_2 t-\lambda_1$ is an upward parabola, so $\{g_\lambda<0\}$ is an interval $(a,b)$ and the minimizer satisfies
\[
    f_*(t)=\frac Bt \mathbbm{1}\{a\le t\le b\}.
\]
The constants \(a,b\) are fixed by the two constraints,
\[
    \int_a^b \frac Bt\,dt=B\log(b/a)=q,
    \qquad
    \int_a^b t\,\frac Bt\,dt=B(b-a)=A.
\]
Rearranging yields
\[
    a=\frac{A}{B(e^{q/B}-1)},
    \qquad
    b=a\,e^{q/B}=\frac{Ae^{q/B}}{B(e^{q/B}-1)}.
\]
The second moment of the extremal distribution is
\[
    \mathbb E X_*^2
    =
    \int_a^b t^2\frac Bt\,dt
    =
    \frac B2(b^2-a^2).
\]
After plugging in the values of $a$ and $b$, a further manipulation yields
\[
    \mathbb E X_*^2
    =
    \frac{A^2}{2B}
    \coth\!\left(\frac{q}{2B}\right),
\]
and hence any such distribution satisfies
\[
    \operatorname{Var}(X)
    \ge
    \operatorname{Var}(X_*)
    =
    \mathbb E X_*^2-A^2
    =
    A^2\left[
        \frac{1}{2B}\coth\!\left(\frac{q}{2B}\right)-1
    \right]
    \ge
    A^2\left[
        \frac{1}{2B}\coth\!\left(\frac{1}{2B}\right)-1
    \right],
\]
the final inequality because \(q\le 1\) and \(\coth\) is decreasing. If \(B\ge 1\), the lower bound $1/(13B^2)$ then follows by monotonicity of $u^{-2}(u\coth u - 1)$. 
\qed


\subsection{Proof of Proposition \ref{prop:master-bound} (master bound)} \label{sec:prop1-proofs}

Let $Z \in \mathbb{R}^p$ be a Gaussian random vector with covariance $\Sigma$ of rank $s \le p$. Then, we may write $Z = Ag$ for $g$ a standard Gaussian vector in $\mathbb{R}^{s}$, where $A$ is a $p \times s$ real matrix such that $AA^*=\Sigma$.
Note that the $i^\text{th}$ row of $A$, which we denote $a_i$, satisfies $\|a_i\|_2^2 = \sigma_i^2$, and we write its normalization as $\nu_i = a_i/\|a_i\|_2$. We have for any $t, \delta > 0$,
\begin{align}
\mathbb{P}\left\{\max_{1 \le i\le p}Z_i > t + \delta\right\}
&= \mathbb{P}\left\{\max_{1 \le i\le p} \bk{a_i}{g} > t + \delta\right\} \nonumber \\
&=\mathbb{P}\left\{\max_{1 \le i\le p} \bk{\frac{a_i}{\|a_i\|_2}}{g} > \frac{t}{\|a_i\|_2} + \frac{\delta}{\|a_i\|_2}\right\} \nonumber \\
&=
\mathbb{P}\left\{
\max_{1 \le i\le p} \bk{\nu_i}{g} > \sigma_i^{-1}t
+\sigma_i^{-1}\delta\right\}. \label{eq:standardize-polytope}
\end{align}

 Without loss of generality, we may assume that all $\sigma_i$ are strictly positive, as the remaining coordinates can be dropped from the event in \eqref{eq:standardize-polytope}; since the final bound increases in the number of inequalities, $p$, this is harmless. We use the following Lemma \ref{lem:derivative-reduction}, mirroring the standard reduction to Nazarov's inequality in \citep{chernozhukov2016empirical}. For completeness, its proof is given in the appendix.

\begin{lemma}\label{lem:derivative-reduction}
    Let $K \subset \mathbb{R}^{s}$ denote the polytope defined by the inequalities $ \bk{\nu_i}{x} \le t/\sigma_i$ for $1 \le i \le p$, and let $F_i \subset \partial K$ be the facet contained in the hyperplane $H_i$ defined by $\bk{\nu_i}{x} = t/\sigma_i$. Then, for any $t, \delta > 0$,
    \begin{equation}\label{eq:inhomogeneous-decomp}
        \lim_{\delta \downarrow 0} \frac{1}{\delta}\left( \bb{P}\left\{\max_{1 \le i\le p}Z_i \le t + \delta\right\}
- \bb{P}\left\{\max_{1 \le i\le p}Z_i \le t \right\} \right) \le \sum_{i=1}^p \frac{1}{\sigma_i} \int_{F_i} \vp_s(x) \mathscr{H}^{s-1}(dx)
,
    \end{equation}
where $\mathscr{H}^{s-1}$ is the $(s-1)$-dimensional Hausdorff measure on  $\partial K \subset \mathbb{R}^{s}$.
\end{lemma}
\begin{proof}
    See Appendix \ref{apx:proofs}.
\end{proof}

 The crucial observation in the proof is that the facets $F_i$ in \eqref{eq:inhomogeneous-decomp} are far away from the origin {precisely when} the corresponding $\sigma_i$ are small. Accounting for this correspondence drastically improves the final bound. 
 
 Following \cite{chernozhukov2017detailed}, for a set $S \subset \mathbb{R}^s$, we may define $d(x,S) = \inf_{y \in S} \|x-y\|$ and, for any facet $F_i$ as defined in Lemma \ref{lem:derivative-reduction}, we define its outward normal column as
\[N_i = \Set[x + y\nu_i]{x \in \mathrm{relint}(F_i), y > 0}.\]
We also have the following estimates, which are taken from the proof of Nazarov's inequality given by \cite{chernozhukov2017detailed}. Details are provided in Appendix \ref{apx:proofs}.

\begin{lemma}\label{lem:disjoint}
The sets $N_i=\Set[x + y\nu_i]{x \in \mathrm{relint}(F_i), y > 0}$ are disjoint, so $\sum_{i}\gamma_s(N_i) \le 1$.
\end{lemma}
\begin{lemma}\label{lem:basic-surface-area-bounds}
For any facet $F_i \subset \partial K$ as defined in Lemma \ref{lem:derivative-reduction}, we have
\begin{align}
    \int_{F_i} \vp_s(x) \mathscr{H}^{s-1}(dx)  \le [1+d(0,H_i)]\gamma_s(N_i)\, \t{ and } \label{eq:conic-bound}\\
    \int_{F_i} \vp_s(x) \mathscr{H}^{s-1}(dx)  \le \frac{1}{\sqrt{2\pi}}\exp\left\{-\frac{d(0,H_i)^2}{2}\right\}.\label{eq:isoperimetric-bound}
\end{align}
\end{lemma}
Finally, we apply the main novel idea in the proof: facets
$F_i$ for which $\sigma_i$ is small (corresponding to $i \in S_1$ below) cannot contribute much to the sum on the right hand side of inequality \eqref{eq:inhomogeneous-decomp}.
To do so, we partition $\sigma_i$ into two complementary subsets $S_1(u)$ and $S_2(u)$ defined by:
\begin{align*}
    S_1(u) &\coloneqq \set{i \in [p]: \sigma_i \le u}; \\
    S_2(u) &\coloneqq \set{i \in [p]:  \sigma_i > u }.
\end{align*}
Crucially, note that the distance between the hyperplane $H_i$ and the origin satisfies
\[d(0,H_i) =t/{\sigma_i}\]
by construction. For $i \in S_1(u)$, inequality \eqref{eq:isoperimetric-bound} implies that
\begin{align*}
   \frac{1}{\sigma_i}\int_{F_i} \varphi_s(x) \mathscr{H}^{s-1}(dx) \le \frac{1}{\sigma_i\sqrt{2\pi}}\exp\left\{-\frac{t^2}{2\sigma_i^2}\right\} = \frac{1}{\sigma_i}\vp\left(\frac{t}{\sigma_i}\right),
\end{align*}
so that
\[\sum_{i\in S_1(u)}\frac{1}{\sigma_i}\int_{F_i} \varphi_s(x) \mathscr{H}^{s-1}(dx) \le \sum_{i\in S_1(u)}\frac{1}{\sigma_i}\vp\left(\frac{t}{\sigma_i}\right).\]
Meanwhile, for $i \in S_2(u)$, it follows by \eqref{eq:conic-bound} that
\begin{align*}
\sum_{i\in S_2(u)}\frac{1}{\sigma_i}\int_{F_i} \varphi_s(x) \mathscr{H}^{s-1}(dx)
&\le \sum_{i\in S_2(u)} \frac{1}{\sigma_i}[1 + d(0,H_i)]\gamma_s(N_i) \\
&\le  \frac{u + t}{u^2} \sum_{i\in S_2(u)} \gamma_s(N_i) \le \frac{u + t}{u^2}.
\end{align*}
This completes the proof.


\section{Conclusion}\label{sec:conclusion}

This paper established a scale-free density bound for the maximum of a centered Gaussian vector: on $(0,\infty)$ the density satisfies $f(t) \le 4\log p / t$, with no restriction on the covariance. From this we derived uniform anti-concentration and variance bounds for the maximum, control of its density at extremal quantiles, and Gaussian and bootstrap approximation guarantees for high-dimensional maxima that remain valid under degeneracy.

Our density bounds enter the approximation results of Section~\ref{sec:applications} only as a drop-in replacement for Nazarov's inequality within the coupling of \citet{koike2021notes}; it would be of interest to incorporate our anti-concentration results more intrinsically into Gaussian comparison and coupling arguments, as was done by \citep{chernozhukov2022improved} in the non-degenerate case.

\bibliographystyle{plainnat}
\bibliography{bibliography}



\begin{appendix}

    \onehalfspacing

\section{Additional proofs of density bounds}\label{apx:proofs}

\subsection{Proofs from Section \ref{sec:prop1-proofs} (Proposition \ref{prop:master-bound})}

We provide proofs of the various lemmas used in proving Proposition \ref{prop:master-bound}. These results are fairly standard, and the proofs follow closely after \citet{chernozhukov2017detailed}.

\subsubsection{Proof of Lemma \ref{lem:derivative-reduction}}

We begin by restating Lemma \ref{lem:derivative-reduction} to keep track of its dependence on the parameter $t\in \mathbb{R}$.

\begin{lemma}\label{lem:derivative-reduction-apx}
    For $t > 0$, let $K^t \subset \mathbb{R}^s$ denote the polytope defined by the inequalities $ \bk{\nu_i}{x} \le t/\sigma_i$ for $1 \le i \le p$ and let $F_i^t \subset \partial K^t$ be the (possibly redundant) facet contained in the hyperplane $H_i^t$ defined by $\bk{\nu_i}{x} = t/\sigma_i$. Then
    \begin{equation}\label{eq:inhomogeneous-decomp-apx}
        \lim_{\delta \downarrow 0} \frac{1}{\delta}\left( \bb{P}\left\{\max_{1 \le i\le p}Z_i \le t + \delta\right\}
- \bb{P}\left\{\max_{1 \le i\le p}Z_i \le t \right\} \right) \le \sum_{i=1}^p \frac{1}{\sigma_i} \int_{F_i^t} \vp_s(x)\, \mathscr{H}^{s-1}(dx)
,
    \end{equation}
where $\mathscr{H}^{s-1}$ denotes the $(s-1)$-dimensional Hausdorff measure on  $\partial K^t \subset \mathbb{R}^s$.
\end{lemma}



\proof[Proof of Lemma \ref{lem:derivative-reduction-apx}]{First, put $u(x) = \max_{1 \le i \le p} \bk{a_i}{x}$; following the discussion at the beginning of Section \ref{sec:prop1-proofs}, the law of $\max_{1 \le i \le p} Z_i$ coincides with that of $u(g)$ for a standard Gaussian vector $g \in \mathbb{R}^s$. Note that $u$ is Lipschitz, hence differentiable Lebesgue-almost everywhere by Rademacher's theorem; this is all the coarea formula below requires. The boundary $\partial K^t$ coincides with the level set $\{x\mid u(x) = t\}$, and, on the interior of the facet $F_i^t \subset K^t$, the gradient of $u$ satisfies $\nabla u = a_i = \nu_i\sigma_i$ where $\nu_i$ is the outward unit normal vector to $F_i^t$. Thus the problem reduces to computing 
\begin{align*}
    \frac{d}{dt} \mathbb{P}\{u(g) \le t\} &= \frac{d}{dt} \int_{\mathbb{R}^s} \mathbbm{1}\{u(x) \le t\}\varphi_s(x)\,dx.
\intertext{Thus, the standard coarea integration formula {\citep[Theorem 3.2.12]{federer2014geometric}} gives}
  &= \frac{d}{dt} \int_{-\infty}^t \int_{\partial K^r} \frac{\varphi_s(x)}{|\nabla u(x)|}\,\mathscr{H}^{s-1}(dx)\,dr.
  \intertext{By the fundamental theorem of calculus, this simplifies to}
  &=  \int_{\partial K^t} \frac{\varphi_s(x)}{|\nabla u(x)|}\,\mathscr{H}^{s-1}(dx).
\intertext{Finally, since the intersection of any two facets has $(s-1)$-dimensional Hausdorff measure zero, this becomes a sum over non-redundant facets $F_i^t$ (indexed by $I \subset [p]$)}
    &=  \sum_{i \in I} \frac{1}{\sigma_i} \int_{F_i^t} \varphi_s(x)\,\mathscr{H}^{s-1}(dx).
\end{align*}
Finally, we may add back redundant facets, as this only increases the sum.
}
\subsubsection{Proofs of Lemmas \ref{lem:disjoint} and \ref{lem:basic-surface-area-bounds}}

\smallskip

\proof[Proof of Lemma \ref{lem:disjoint}]{
For any $x \in N_i$ we may write $x = u + y \nu_i$ for $y > 0$ and $u \in \mathrm{relint}(F_i)$. Therefore, we have
\[\bk{x}{\nu_i} = \bk{u}{\nu_i} + y = \frac{t}{\sigma_i} + y > \frac{t}{\sigma_i} ,\] from which we conclude that the hyperplane $H_i=\{x: \bk{\nu_i}{x}=t/\sigma_i\}$ separates $K$ from $x$. Moreover, for any $z \in K$ we have $\bk{x-u}{z-u} = y\,\bk{\nu_i}{z-u} = y\big(\bk{\nu_i}{z} - t/\sigma_i\big) \le 0$, since $\bk{\nu_i}{z} \le t/\sigma_i = \bk{\nu_i}{u}$ for $z \in K$; this is precisely the variational characterization of the Euclidean projection, so $u$ is the unique nearest point of $K$ to $x$. Consequently, if $x \in N_i \cap N_j$ for distinct facets, its projection $u$ would lie in both $\mathrm{relint}(F_i)$ and $\mathrm{relint}(F_j)$, which is impossible; hence the $N_i$ are disjoint.}

\proof[Proof of Lemma \ref{lem:basic-surface-area-bounds}]{
The bound \eqref{eq:conic-bound} is due to \citet{nazarov2004maximal}, and is well-known (see also equation (5) of \citet{klivans2008learning}). As stated, it is implied by Lemma 2 of \citet{chernozhukov2017detailed}.

We prove \eqref{eq:isoperimetric-bound}; essentially, the Gaussian surface area of $F_i$ is smaller than that of the hyperplane $H_i$, which is well-known. Notice that both $\varphi_s(-)$ and $\mathscr{H}^{s-1}(-)$ are invariant to orthogonal transformations of $\mathbb{R}^s$. Therefore, without any loss of generality, we can assume that $H_i$ has the parameterization $\ (t/\sigma_i, u_1, \ldots, u_{s-1})$ for $u \in \mathbb{R}^{s-1}$, with $F_i$ corresponding to a subset $U_i \subset  \mathbb{R}^{s-1}$, and $N_i$ corresponding to $(t/\sigma_i + r, u_1, \ldots, u_{s-1})$ for $u \in U_i$ and $r > 0$. Then
\begin{align*}
\int_{F_i} \varphi_s(y) \mathscr{H}^{s-1}(dy)
&\le \int_{H_i} \varphi_s(x) \mathscr{H}^{s-1}(dx) \\
\intertext{since $F_i \subset H_i$ and the integrand is positive. Given our parameterization of $H_i$, this has the form}
&= \vp_1(t/\sigma_i) \int_{\mathbb{R}^{s-1}} \vp_{s-1}(u)\,du = \vp_1(t/\sigma_i).
\end{align*}
Since $d(0,H_i) = t/\sigma_i$ by construction, \eqref{eq:isoperimetric-bound} is proved.
}

\section{Proofs of approximation results}\label{sec:application-proofs}

The Gaussian and bootstrap approximation results stated in Section \ref{sec:applications} follow by combining the density bounds of Section \ref{sec:main-technical} with the coupling results of \citet{koike2021notes}, in the notation fixed at the start of Section \ref{sec:applications}.

\subsection{One-sided couplings}\label{sec:onesided-couplings}

The proofs rely on the following one-sided Prokhorov-type couplings of \citet{koike2021notes}, which we restate for convenience. The first compares $S_n$ to its Gaussian analog $Z \sim N(0,\Sigma)$.

\begin{lemma}[Lemma 5.7 of \citet{koike2021notes}]\label{lem:onesided-prokhorov}
Assume that there exists a constant \(K \ge 1\) such that for all $1 \le i \le n$ and $1 \le j \le p$, \(\|X_{ij}\|_{\psi_1} \le K.\)
Then there exists a universal constant \(C>0\) such that for every
Borel set \(A \subset \mathbb{R}\), every \(y \in \mathbb{R}^p\), and every
\(
\varepsilon \ge CK^2 n^{-\nicefrac{1}{2}}(\log n)(\log p),
\)
we have
\begin{equation}\label{eq:gaussian-onesided}
\begin{split}
\mathbb{P}\!\left\{\max_{1 \le j \le p}(S_{n})_j - y_j \in A\right\}
&\le \mathbb{P}\!\left\{\max_{1 \le j \le p} Z_j - y_j \in A^{6\varepsilon}\right\}\\
&\qquad + C\varepsilon^{-2}\!\left(\frac{K^{2}(\log p)^{3/2}}{\sqrt{n}} + \frac{K^{4}(\log p)^2(\log n)^2}{n}\right),
\end{split}
\end{equation}
where $A^\delta = \cup_{a\in A} [a-\delta,a+\delta]$ denotes the $\delta$-enlargement of $A$.
\end{lemma}

\begin{remark}\label{rmk:Bn}
Koike's bound is stated in terms of a constant $B_n \ge 1$ bounding both $\max_{i,j}\|X_{ij}\|_{\psi_1}$ and $(\max_j n^{-1}\sum_i \mathbb{E}[X_{ij}^4])^{1/2}$. We bound the fourth moment using the $\psi_1$-norm, so $B_n \asymp K^2$; this is the source of the powers $K^2$ and $K^4$ in \eqref{eq:gaussian-onesided}.
\end{remark}

The corresponding one-sided coupling for the multiplier bootstrap statistic $S_n^{\mathrm{WB}}$ was also derived by \citet{koike2021notes}.

\begin{lemma}[Bootstrap version of \citealp{koike2021notes}, Thm.~3.1]\label{lem:onesided-bootstrap}
Under the assumptions of Lemma \ref{lem:onesided-prokhorov}, with multipliers $(w_i)_{i=1}^n$ as in Section \ref{sec:app-bootstrap} and $b \ge 1$ such that $|w_i| \le b$ almost surely, there exists a universal constant $C' > 0$ such that, for every $\varepsilon \ge C'bK^2 n^{-\nicefrac{1}{2}}(\log n)(\log p)$, there exist an event $\Omega_\varepsilon$ with $\mathbb{P}(\Omega_\varepsilon) = 1$ and a nonnegative random variable $\rho_\varepsilon(X)$, depending on $X$ but not on the choice of $A$ or $y$, with
\begin{equation}\label{eq:bootstrap-remainder}
\mathbb{E}[\rho_\varepsilon(X)] \le C'\varepsilon^{-2}\!\left(\frac{b\, K^{2}(\log p)^{3/2}}{\sqrt{n}} + \frac{b^2 K^{4}(\log p)^2(\log n)^2}{n}\right),
\end{equation}
such that, on $\Omega_\varepsilon$, for every Borel set $A \subset \mathbb{R}$ and every $y \in \mathbb{R}^p$,
\begin{equation}\label{eq:bootstrap-onesided}
\mathbb{P}\!\left\{\max_{1 \le j \le p}(S_n^{\mathrm{WB}})_j - y_j \in A\,\middle|\, X\right\}
\le \mathbb{P}\!\left\{\max_{1 \le j \le p} Z_j - y_j \in A^{6\varepsilon}\right\} + \rho_\varepsilon(X).
\end{equation}
\end{lemma}
\begin{proof}
The statement is the one-sided, pre-optimization form of \citet[Theorem 3.1(a)]{koike2021notes}; since we use the intermediate coupling rather than its optimized conclusion, we recall the steps from the proof in \citet[Sect.~6.1]{koike2021notes}. We abbreviate the remainders
\[
\delta_{n,1} = \frac{B_n^2(\log p)^3}{n}, \qquad \delta_{n,2} = \frac{B_n^2(\log p)^2(\log n)^2}{n},
\]
and use $B_n \asymp K^2$ (Remark \ref{rmk:Bn}); thus $\sqrt{\delta_{n,1}} \asymp K^2(\log p)^{3/2}/\sqrt n$ and $\delta_{n,2} \asymp K^4(\log p)^2(\log n)^2/n$ are the two terms in \eqref{eq:bootstrap-remainder}. Set $\kappa_n = 2K^2\log n$ and, exactly as in the truncation underlying Lemma \ref{lem:onesided-prokhorov} \citep[Lemma 5.7]{koike2021notes}, put $\tilde X_{ij} = X_{ij}\mathbbm{1}\{|X_{ij}|\le\kappa_n\} - \mathbb{E}[X_{ij}\mathbbm{1}\{|X_{ij}|\le\kappa_n\}]$ and $\tilde Y_i = \tilde X_i - \bar{\tilde X}_n$, so that $\max_{i,j}|w_i\tilde Y_{ij}|\le 4b\kappa_n$ almost surely.

\paragraph{\emph{Truncation split.}} Writing $Y_i = X_i - \bar X_n$ (so that $S_n^{wY} := n^{-1/2}\sum_i w_i Y_i = S_n^{\mathrm{WB}}$), for every Borel $A$ and $y \in \mathbb{R}^p$,
\begin{equation}\label{eq:wb-split}
\mathbb{P}\!\left\{\max_j (S_n^{\mathrm{WB}})_j - y_j \in A \,\middle|\, X\right\}
\le \mathbb{P}\!\left\{\max_j (S_n^{w\tilde Y})_j - y_j \in A^{\varepsilon} \,\middle|\, X\right\}
+ \mathbb{P}\!\left\{\|S_n^{w(Y-\tilde Y)}\|_\infty > \varepsilon \,\middle|\, X\right\}.
\end{equation}

\paragraph{\emph{Conditional coupling.}} Conditionally on $X$, the vectors $(w_i\tilde Y_i)_{i=1}^n$ are independent and centered, with conditional covariance $n^{-1}\sum_i \mathbb{E}[w_i^2]\,\tilde Y_i\tilde Y_i^\top = n^{-1}\sum_i \tilde Y_i\tilde Y_i^\top$ since $\mathbb{E}[w_i^2]=1$. Applying the coupling of \citet[Lemma 4.1]{koike2021notes}\textemdash proved through the third-moment-matching randomized Lindeberg interpolation of \citet[Proposition 4.1]{koike2021notes} together with a Stein-kernel estimate\textemdash to $(w_i\tilde Y_i)$ given $X$, with Gaussian target $Z \sim N(0,\Sigma)$, $\Sigma = \mathbb{E}[S_n S_n^\top]$, yields, for each fixed $\varepsilon$ above the stated threshold, a probability-one event $\Omega_\varepsilon$ on which the following holds for every Borel $A$ and every $y$:
\begin{equation}\label{eq:wb-conditional}
\mathbb{P}\!\left\{\max_j (S_n^{w\tilde Y})_j - y_j \in A^{\varepsilon} \,\middle|\, X\right\}
\le \mathbb{P}\!\left\{\max_j Z_j - y_j \in A^{6\varepsilon}\right\}
+ C\varepsilon^{-2}\!\left(\Delta^*_{n,0}\log p + \Delta^*_{n,1}\sqrt{\tfrac{(\log p)^3}{n}}\right),
\end{equation}
where the $5\varepsilon$ enlargement of \citet[Lemma 4.1]{koike2021notes} turns $A^\varepsilon$ into $A^{6\varepsilon}$, and the conditional remainder quantities are
\[
\Delta^*_{n,0} = \mathbb{E}\!\left[\max_{1\le j,k\le p}\Big|\tfrac1n\textstyle\sum_i w_i^2\tilde Y_{ij}\tilde Y_{ik} - \Sigma_{jk}\Big| \,\middle|\, X\right],
\qquad
\Delta^*_{n,1} = \left(\tfrac1n\,\mathbb{E}\!\left[\max_j \textstyle\sum_i w_i^4\tilde Y_{ij}^4 \,\middle|\, X\right]\right)^{1/2}.
\]
Crucially, $\Delta^*_{n,0}$ and $\Delta^*_{n,1}$ depend on $X$ but not on $(A,y)$, so \eqref{eq:wb-conditional} holds for all $(A,y)$ simultaneously on $\Omega_\varepsilon$ with an $A$-independent error; this is what later permits taking a supremum over $t$ before taking the expectation.

\paragraph{\emph{Truncation error.}} Since $|w_i|\le b$ and $\sqrt n(\bar X_n - \bar{\tilde X}_n) = S_n^{X-\tilde X}$, the second term of \eqref{eq:wb-split} is bounded exactly as the truncation estimate underlying Lemma \ref{lem:onesided-prokhorov} \citep[Eq.~(25) and Sect.~6.1]{koike2021notes}:
\[
\mathbb{P}\!\left\{\|S_n^{w(Y-\tilde Y)}\|_\infty > \varepsilon \,\middle|\, X\right\} \le C\varepsilon^{-2} b^2\delta_{n,2}.
\]
Combining the three displays gives \eqref{eq:bootstrap-onesided} on $\Omega_\varepsilon$ with the nonnegative, $A$-independent remainder
\[
\rho_\varepsilon(X) = \mathbb{P}\!\left\{\|S_n^{w(Y-\tilde Y)}\|_\infty > \varepsilon \,\middle|\, X\right\} + C\varepsilon^{-2}\!\left(\Delta^*_{n,0}\log p + \Delta^*_{n,1}\sqrt{\tfrac{(\log p)^3}{n}}\right).
\]
The threshold $\varepsilon \ge C'bK^2 n^{-1/2}(\log n)(\log p)$ corresponds to the requirement $\varepsilon \ge 12\,b\kappa_n(\log p)/\sqrt n$, ensuring the truncation level is compatible with the coupling.

\paragraph{\emph{Remainder in expectation.}} Taking expectations over $X$ bounds each conditional term. As $\mathbb{E}[w_i^2]=1$, the conditional mean of $n^{-1}\sum_i w_i^2\tilde Y_{ij}\tilde Y_{ik}$ is the truncated sample covariance, whose deviation from $\Sigma$ is controlled by the fourth-moment concentration of \citet[Lemma 5.5]{koike2021notes} and a maximal inequality for $\|\bar X_n\|_\infty$; the computation of \citet[Sect.~6.1]{koike2021notes} gives $\mathbb{E}[\Delta^*_{n,0}]\log p \lesssim b\sqrt{\delta_{n,1}} + b^2\delta_{n,2}$. For $\Delta^*_{n,1}$, the bound $\mathbb{E}[w_i^4]\le b^2\mathbb{E}[w_i^2] = b^2$ together with the Jensen and Lyapunov inequalities reduces it to the Gaussian fourth-moment term, giving $\mathbb{E}[\Delta^*_{n,1}]\sqrt{(\log p)^3/n}\lesssim b\sqrt{\delta_{n,1}} + b^2\delta_{n,2}$. Collecting the three contributions yields $\mathbb{E}[\rho_\varepsilon(X)] \le C'\varepsilon^{-2}(b\sqrt{\delta_{n,1}} + b^2\delta_{n,2})$, which is \eqref{eq:bootstrap-remainder}.
\end{proof}

\subsection{Proof of Corollaries \ref{cor:nonuniform-clt}--\ref{cor:uniform-bootstrap}}\label{sec:proof-clt-bootstrap}

The signed corollaries are proved by combining the one-sided couplings of Lemmas \ref{lem:onesided-prokhorov}--\ref{lem:onesided-bootstrap} with the density bounds of Section \ref{sec:main-technical}; the unsigned corollaries are treated at the end. We write the argument for the Gaussian case; the bootstrap argument is identical after replacing Lemma \ref{lem:onesided-prokhorov} with Lemma \ref{lem:onesided-bootstrap} and tracking the additional factor of $b$.

\smallskip

\noindent\emph{Step 1: a two-sided Prokhorov bound.}
Lemma \ref{lem:onesided-prokhorov} provides a one-sided bound; to derive a two-sided bound on a CDF difference, we apply it twice with $y = 0$. First, taking $A = (-\infty, t]$ in \eqref{eq:gaussian-onesided}, we obtain
\begin{equation}\label{eq:clt-onesided-upper}
\mathbb{P}\!\left\{\max_j (S_n)_j \le t\right\} - \mathbb{P}\!\left\{\max_j Z_j \le t\right\}
\le \mathbb{P}\!\left\{t < \max_j Z_j \le t + 6\varepsilon\right\} + R_n(\varepsilon),
\end{equation}
where $R_n(\varepsilon) := C\varepsilon^{-2}\!\left(K^{2}(\log p)^{3/2}/\sqrt{n} + K^{4}(\log p)^2(\log n)^2/n\right)$. Second, taking $A = (t,\infty)$, we get
\begin{equation}\label{eq:clt-onesided-lower} \begin{split}
\mathbb{P}\!\left\{\max_j Z_j \le t - 6\varepsilon\right\} - \mathbb{P}\!\left\{\max_j (S_n)_j \le t\right\}
&\le R_n(\varepsilon). \end{split}
\end{equation}
Combining \eqref{eq:clt-onesided-upper} and \eqref{eq:clt-onesided-lower} yields
\begin{equation}\label{eq:two-sided-bound}
\left|\mathbb{P}\!\left\{\max_j (S_n)_j \le t\right\} - \mathbb{P}\!\left\{\max_j Z_j \le t\right\}\right|
\le \mathbb{P}\!\left\{t - 6\varepsilon < \max_j Z_j \le t + 6\varepsilon\right\} + R_n(\varepsilon).
\end{equation}

\smallskip

\noindent\emph{Step 2: anti-concentration of $\max_j Z_j$.}
The first term on the right-hand side of \eqref{eq:two-sided-bound} is controlled by the density of $\max_j Z_j$ on $[t-6\varepsilon, t+6\varepsilon]$. Throughout, $C$ denotes a universal constant whose value may change between occurrences. We distinguish two cases.

\paragraph{\emph{Case I (Corollaries \ref{cor:nonuniform-clt} and \ref{cor:nonuniform-bootstrap}, $t = t_q^\Sigma$ for $q \ge \nicefrac{2}{3}$).}}
By stochastic dominance, $t_q^\Sigma \ge q^M_{\nicefrac{2}{3}} \ge \sigma_{\mathrm{max}}\Phi^{-1}(\nicefrac{2}{3}) \ge \nicefrac{2 \sigma_{\mathrm{max}}}{5}$ for every $q \ge \nicefrac{2}{3}$. Hence, provided $\varepsilon \le \nicefrac{\sigma_{\mathrm{max}}}{40}$, the interval $[t-6\varepsilon, t+6\varepsilon]$ lies above $\nicefrac{\sigma_{\mathrm{max}}}{4}$, and Proposition \ref{prop:general-finite-bound} gives $f(s) \le C\log p/\sigma_{\mathrm{max}}$ there, so
\[
\mathbb{P}\!\left\{t-6\varepsilon < \max_j Z_j \le t+6\varepsilon\right\} \le \frac{C\,\varepsilon \log p}{\sigma_{\mathrm{max}}}.
\]

\paragraph{\emph{Case II (Corollaries \ref{cor:uniform-clt} and \ref{cor:uniform-bootstrap}, $t \in \mathbb{R}$ arbitrary).}}
Splitting $[t-6\varepsilon, t+6\varepsilon]$ at the fixed cut $a = \mu/2$, as in the proof of Corollary \ref{cor:finite-window} but without tying the cut to the window width, the mass below $a$ is at most $e^{-\mu^2/8\sigma_{\mathrm{max}}^2}$ by \eqref{eq:borell-tis}, while above $a$ the density is at most $C\log p/\mu$ by \eqref{eq:general-finite-bound}. Hence
\[
\mathbb{P}\!\left\{t-6\varepsilon < \max_j Z_j \le t+6\varepsilon\right\} \le e^{-\mu^2/8\sigma_{\mathrm{max}}^2} + \frac{C\,\varepsilon\log p}{\mu}.
\]
The additive Gaussian-tail term does not shrink with the coupling scale $\varepsilon$; it is independent of $n$ and depends only on the ratio $\mu/\sigma_{\mathrm{max}}$.

\smallskip

\paragraph{\emph{Step 3: optimization in $\varepsilon$.}}
Write $\nu = \sigma_{\mathrm{max}}$, $P = 0$ in Case I, and $\nu = \mu$, $P = e^{-\mu^2/8\sigma_{\mathrm{max}}^2}$ in Case II. Combining \eqref{eq:two-sided-bound} with Step 2, for every admissible $\varepsilon$,
\begin{equation}\label{eq:pre-optimization}
\left|\mathbb{P}\!\left\{\max_j (S_n)_j \le t\right\} - \mathbb{P}\!\left\{\max_j Z_j \le t\right\}\right|
\le P + C\Big(\frac{\varepsilon \log p}{\nu} + \varepsilon^{-2}\Delta_1 + \varepsilon^{-2}\Delta_2\Big),
\end{equation}
where $\Delta_1 = K^2(\log p)^{3/2}/\sqrt n$ and $\Delta_2 = K^4(\log p)^2(\log n)^2/n$ are the two terms of $R_n(\varepsilon)$. Balancing the first two non-penalty terms gives the optimal scale
\[
\varepsilon^* \asymp \nu^{1/3} K^{2/3} (\log p)^{1/6} n^{-1/6},
\qquad
\frac{\varepsilon^*\log p}{\nu} + \varepsilon^{*-2}\Delta_1 \asymp R := \nu^{-2/3}\Big(\frac{K^4 \log^7 p}{n}\Big)^{1/6}.
\]
Two observations complete the proof. First, the remaining term obeys 
\[\varepsilon^{*-2}\Delta_2 = R \cdot K^2(\log p)^{1/2}(\log n)^2 n^{-1/2} \le R\] by Assumption~\ref{assn:sampling} (with $b = 1$). Second, the optimal $\varepsilon^*$ may be inadmissible\textemdash either $\varepsilon^* > \sigma_{\mathrm{max}}/40$ (violating the requirement of Case~I), or $\varepsilon^*$ below the lower threshold $\asymp\sqrt{\Delta_2}$ of Lemma \ref{lem:onesided-prokhorov}\textemdash but in either case $R$ exceeds a universal positive constant (in the first, $R \ge (\log p)/40$; in the second, Assumption~\ref{assn:sampling} forces $R \gtrsim 1$), so the asserted bound holds trivially because the left-hand side is at most $1$. In all cases,
\[
\left|\mathbb{P}\!\left\{\max_j (S_n)_j \le t\right\} - \mathbb{P}\!\left\{\max_j Z_j \le t\right\}\right| \lesssim P + R.
\]
Taking the supremum over $t$\textemdash over $t \ge q^M_{\nicefrac{2}{3}}$ in Case~I ($P = 0$), and over $t \in \mathbb{R}$ in Case~II\textemdash gives Corollaries \ref{cor:nonuniform-clt} and \ref{cor:uniform-clt}; in Case~II the displayed rate dominates $P$ precisely when $\mu \gtrsim \sigma_{\mathrm{max}}\sqrt{\log n}$.

For the bootstrap corollaries we argue from the coupling of Lemma \ref{lem:onesided-bootstrap}. For each fixed admissible $\varepsilon$, on the probability-one event $\Omega_\varepsilon$, the two-sided bound \eqref{eq:two-sided-bound} holds for every $t$ simultaneously with $R_n(\varepsilon)$ replaced by the $A$-independent remainder $\rho_\varepsilon(X)$. Since $\rho_\varepsilon(X)$ does not depend on $t$, we may take the supremum over $t$ on $\Omega_\varepsilon$ before taking the expectation, after which $\mathbb{E}[\rho_\varepsilon(X)]$ obeys \eqref{eq:bootstrap-remainder}. As the remainder carries an extra factor $b$ in $\Delta_1$ and $b^2$ in $\Delta_2$, the same optimization yields $\sigma_{\mathrm{max}}^{-2/3}(b^2 K^4\log^7 p/n)^{1/6}$, and the analogue with $\mu$, giving Corollaries \ref{cor:nonuniform-bootstrap} and \ref{cor:uniform-bootstrap}.

\smallskip
\noindent\emph{Unsigned maxima (Corollaries \ref{cor:uniform-clt-unsigned} and \ref{cor:uniform-bootstrap-unsigned}).}
Apply the one-sided couplings to the reflected vectors $(X_i, -X_i)$ (and $w_i(X_i - \bar X_n)$ likewise) in dimension $2p$, whose coordinate maxima are $\max_j|(S_n)_j|$ and $\max_j|(S_n^{\mathrm{WB}})_j|$ with Gaussian analog $\max_j|Z_j|$. Put $L_p=\log(2p)$ and let $c=1$ in the Gaussian case and $c=b$ in the bootstrap case. The two coupling remainders are then
\[
\Delta_1(c)=\frac{cK^2L_p^{3/2}}{\sqrt n},
\qquad
\Delta_2(c)=\frac{c^2K^4L_p^2(\log n)^2}{n}.
\]
By Proposition \ref{prop:unsigned-density-bound}, uniformly in $t\in\mathbb R$,
\[
\mathbb{P}\Big\{t - 6\varepsilon < \max_j|Z_j| \le t + 6\varepsilon\Big\}
\;\lesssim\;
\frac{\varepsilon L_p}{\sigma_{\mathrm{max}}}.
\]
Balancing this term against $\varepsilon^{-2}\Delta_1(c)$ gives
\(
\varepsilon^*
\asymp
\sigma_{\mathrm{max}}^{1/3}c^{1/3}K^{2/3}L_p^{1/6}n^{-1/6}
\)
and
\[
\frac{\varepsilon^*L_p}{\sigma_{\mathrm{max}}}
+(\varepsilon^*)^{-2}\Delta_1(c)
\asymp
\sigma_{\mathrm{max}}^{-2/3}
\left(\frac{c^2K^4L_p^7}{n}\right)^{1/6}
=:R^*.
\]
Assumption~\ref{assn:sampling} gives $\Delta_2(c)\lesssim\Delta_1(c)$. If $\varepsilon^*$ is below the coupling threshold $cK^2n^{-1/2}(\log n)L_p$, direct substitution gives
\[
R^*
\gtrsim
\frac{\sqrt n}{cK^2L_p^{1/2}(\log n)^2}
\gtrsim1
\]
by Assumption~\ref{assn:sampling}, so the bound is trivial; otherwise $\varepsilon^*$ is admissible and gives $R^*$. For the bootstrap result, take the expectation of the $t$-uniform conditional remainder exactly as in the signed case. \qed

\end{appendix}

\end{document}